\documentclass[a4paper]{scrartcl}
\usepackage{amssymb}
\usepackage{amsmath}
\usepackage{amsthm}
\usepackage{mathabx}
\usepackage{latexsym}
\usepackage{enumerate}
\usepackage{graphicx,color}
\usepackage{hyperref}

\usepackage{dsfont}

\usepackage{xcolor}
\usepackage{enumitem}

\usepackage[autostyle=true]{csquotes} 

\renewcommand{\d}{\mathrm{d}}
\newcommand{\E}{\mathbb{E}}

\newcommand*{\largecdot}{\raisebox{-0.25ex}{\scalebox{1.4}{$\cdot$}}}

\makeatletter
\newcommand*\bigcdot{\mathpalette\bigcdot@{.5}}
\newcommand*\bigcdot@[2]{\mathbin{\vcenter{\hbox{\scalebox{#2}{$\m@th#1\bullet$}}}}}
\makeatother

\numberwithin{equation}{section}

\numberwithin{figure}{section}

\theoremstyle{plain}
\newtheorem{theorem}{Theorem}[section]
\newtheorem{proposition}[theorem]{Proposition}
\newtheorem{corollary}[theorem]{Corollary}

\theoremstyle{definition}
\newtheorem{definition}[theorem]{Definition}
\newtheorem{assumption}[theorem]{Assumption}

\newtheorem{example}[theorem]{Example}
\newtheorem{remark}[theorem]{Remark}
\AtBeginEnvironment{example}{%
  \pushQED{\qed}%
}
\AtEndEnvironment{example}{\popQED\endexample}
\AtBeginEnvironment{remark}{%
  \pushQED{\qed}%
}
\AtEndEnvironment{remark}{\popQED\endexample}

\usepackage{authblk}

\title{Canonical insurance models: stochastic equations and comparison theorems}

\author[1]{Marcus C.~Christiansen}
\author[2,$\star$]{Christian Furrer}

\affil[1]{\footnotesize Institut f{\"u}r Mathematik, Carl von Ossietzky Universit{\"a}t Oldenburg, Germany.}
\affil[2]{\footnotesize Department of Mathematical Sciences, University of Copenhagen, Denmark.}

\affil[$\star$]{\footnotesize Corresponding author. E-mail: \href{mailto:furrer@math.ku.dk}{furrer@math.ku.dk}.}

\date{}

\begin{document}

\maketitle

\begin{abstract}

Thiele's differential equation explains the change in prospective reserve and plays a fundamental role in safe-side calculations and other types of actuarial model comparisons. This paper presents a `model lean' version of Thiele's equation with the novel feature that it supports any canonical insurance model, irrespective of the model's intertemporal dependence structure. The basis for this is a canonical and path-wise model construction that simultaneously handles discrete and absolutely continuous modeling regimes. Comparison theorems for differing canonical insurance models follow directly from the resulting stochastic backward equations. The elegance with which these comparison theorems handle non-equivalence of probability measures is one of their major advantages over previous results. 

\end{abstract}

\vspace{5mm}


\vspace{5mm}

\section{Introduction}

In actuarial science, safe-side calculations for the prospective reserve based on prudent actuarial bases date back to at least~\cite{Lidstone1905}. In modern times, safe-side calculation results for survival models~\cite{Norberg1985} have been extended to first Markov models~\cite{Hoem1988,RamlauHansen1988,Linnemann1993} and later semi-Markov models~\cite{Niemeyer2015}. The main idea is to carefully analyze the dynamics of the so-called state-wise prospective reserves, described by the celebrated Thiele equation, which allows for the comparison of differing insurance models. Beyond safe-side calculations, such comparison results are essential to modern life insurance mathematics in both resolving the circularity that arises from (implicitly defined) reserve-dependent payments and in the development of efficient computational schemes~\cite{Cantelli1914,Norberg1991,MilbrodtStracke1997,MilbrodtHelbig2008,ChristiansenDenuitDhaene2014}.

Another emerging theme in stochastics relates to model uncertainty and robustness. There is an increasing interest in `model lean' or `model free' approaches. In actuarial multi-state modeling, this is reflected by a recent movement from Markov over semi-Markov~\cite{Christiansen2012,BuchardtMollerSchmidt2015,AhmadBladtFurrer2023,BladtMincaPeralta2023}towards so-called non-Markov modeling~\cite{ChristiansenFurrer2021,Christiansen2021b,Furrer2022,BathkeFurrer2024}. A similar tendency can be observed in the biostatistics literature~\cite{PutterSpitoni2018,Overgaard2019b,Maltzahn2021,Niessl2023}.

In this paper, we derive actuarial comparison theorems for finite state space models with no restrictions whatsoever on the intertemporal dependence structure. In that sense, our approach is `model lean'. Given a canonical insurance model $(\alpha,\Lambda,\Phi,B,b)$ in accordance with Definition~\ref{def:canonical_insurance_model}, which consists of an initial distribution $\alpha$, transition rates $\Lambda$, interest rates $\Phi$, sojourn payments $B$, and transition payments $b$, we show that the corresponding state-wise prospective reserves $(V^i)$ uniquely and surely (path-wise) solve the backward equation
\begin{align*}
0=\mathds{1}_{\{  \Lambda^{i\largecdot}(t-)< \infty\}} I^i(t-)  \bigg(V^i(\d t)   +  B^i(\d t) - V^i(t-) \Phi^i (\d t) + \sum_{j : j\neq i}(b^{ij}(t) + V^j(t)-V^i(t)) \Lambda^{ij}(\d t)\bigg),
\end{align*}
with boundary condition $V^i(T) = 0$ and with $I^i$ indicating whether the process is currently in state $i$, see Theorem~\ref{stochThiele}. This stochastic Thiele equation is then used to establish a wide range of actuarial comparison theorems. Classic results for Markov processes are fully recovered and further refined.

The first stochastic Thiele equation may be found in~\cite{Norberg1992}, but with no clarification on uniqueness of the solution and lack of clarity in the associated definition of state-wise prospective reserves. The latter was already noted in~\cite{Norberg1996}, but not rigorously resolved until investigated in~\cite{ChristiansenFurrer2021}. However, the stochastic Thiele equation proposed in~\cite{ChristiansenFurrer2021} is only an almost sure equation, which can be problematic for comparisons between non-equivalent probabilistic models (actuarial bases). Contrary to~\cite{ChristiansenFurrer2021}, this paper adopts a canonical approach which allows for path-wise statements.

Our canonical approach takes its inspiration from~\cite{Jacobsen2006}, which has hitherho received limited attention in the actuarial literature, see however~\cite{Furrer2022}. Different from \cite{Jacobsen2006}, we take not the distribution of the marked point process or even the compensators of the multivariate counting process as the starting point, but rather so-called (cumulative) transition rates. This comes with its own technical intricacies, but has the advantage that it offers a natural connection to Markov modeling -- corresponding to deterministic transition rates -- as well as the industry practice of composing and applying actuarial risk tables. The approach covers {absolutely continuous as well as discrete} modeling; both are currently common. {Central to both approaches, Jacobsen's and ours, are sure representations for càdlàg martingales. However, Jacobsen takes the existence of a càdlàg modification as given, while we offer an actual construction.}

In~\cite{ChristiansenDjehiche2020} an alternative approach based on backward stochastic differential equations is pursued, also aiming at resolving the circularity arising from implicitly defined payments. The advantage of the canonical approach pursued here is the elegance with which it handles non-equivalent probability measures. The backward stochastic differential equation literature contains a range of comparison theorems{, see for example~\cite{CohenElliot2012}, and the results of this paper may be seen in light of these, but are tailored to the situation with non-equivalent rather than just dominating measures. This is not just to satisfy mathematical curiosity: non-equivalence occurs frequently in the comparison of actuarial models, with the actuary starting from a simpler (non-dominating) model that rules out the occurrence of some event. This could include policyholder behavior events such as surrender, also called lapse, or retirement.} Furthermore, any approach based on backward stochastic differential equations is also limited by the lack of background results outside the absolutely continuous case.

The paper is structured as follows. Sections~\ref{sec:canonical_measurable_space}--\ref{sec:canonical_prob_model} provide the construction of canonical probability models generated by (cumulative) transition rates, culminating with Theorem~\ref{TheoremCanoncialSpace}. The following two  sections contain key technical results. Section~\ref{sec:kolmogorov} deals with the time-dynamics of conditional expectations of state indicator processes, {leading to} a stochastic version of the Kolmogorov backward equation, see Theorem~\ref{GeneralizKolmogBackwardEq}, while Section~\ref{sec:statewise} concerns the extension from state indicator processes to a wider class of processes. In the final two sections, the narrative shifts towards actuarial science with Section~\ref{sec:thiele} devoted to the stochastic Thiele equation and Section~\ref{sec:comparison} to comparison theorems.

\section{Canonical measurable space}\label{sec:canonical_measurable_space}

Status data from an individual life insurance policy is usually of the form
\begin{center}
\begin{tabular}{c|c}
	\hline
	date & status \\
	\hline
$t_0$ & $z_0$ \\
$t_1$ & $z_1$ \\
$t_2$ & $z_2$ \\
$\vdots $& $\vdots$\\
\hline
\end{tabular}
\end{center}
with ordered time points $t_0< t_1 < t_2 <\cdots $ from the time set $[0,\infty)$ and states $z_0 \neq z_1 \neq z_2 \neq \cdots $  from a finite state space $\mathcal{Z}$. By convention, let $t_0=0$ be the starting time of the individual insurance contract. {Further, without loss of generality let $\mathcal{Z}\subset \mathbb{N}$.} The total number of status updates may be countably infinite on the full time line, but on bounded time intervals the number of status updates shall be at most finite. For a convenient notation, if there are only finitely many updates in total, we extend the update sequence to a countably infinite sequence by adding artificial data points  $(\infty, \nabla)$. Let $\bar{\mathcal{Z}}:= \mathcal{Z}\cup \{\nabla\}$. All in all, the  set of potential status developments is
\begin{align*}
	\Omega :=\big\{ (t_k,z_k)_{k \in \mathbb{N}_0}  \in [0,\infty]^{\mathbb{N}_0}  \times \bar{\mathcal{Z}}^{\mathbb{N}_0}  :  t_0&=0,\, t_k < t_{k+1} \textrm{ for } t_k< \infty, t_k=t_{k+1} \textrm{ for }t_k= \infty,\\  z_k& \neq z_{k+1}  \textrm{ for } t_k < \infty , z_k =\nabla  \textrm{ for } t_k =\infty ,\,\sup_k t_k =\infty\}.
\end{align*}
The projection mappings
\begin{align*}
	&\tau_k : \Omega \rightarrow [0,\infty], \quad  \tau_k(\omega) \mapsto t_k, \\
	&\zeta_k : \Omega \rightarrow \bar{\mathcal{Z}}, \quad \zeta_k(\omega) \mapsto z_k, 
\end{align*}
 define a \textbf{marked point process}
\begin{align*}
	(\tau_k, \zeta_k)_{k \in \mathbb{N}_0}.
\end{align*}
An alternative way of representing the insurance data is the \textbf{multivariate counting process} 
\begin{align*}
N = (N^{ij})_{i,j \in \mathcal{Z}: i \neq j}, \quad	N^{ij}: [0,\infty) \times \Omega \rightarrow \mathbb{N}_0,
\end{align*}
defined by
\begin{align*}
	N^{ij}(t) :=  \sum_{ k \in \mathbb{N} } \mathds{1}_{\{ \tau_k \leq t, \zeta_{k-1}=i, \zeta_{k}=j\}}.
\end{align*}
The multivariate counting process $N$ combined with the initial state $\zeta_0$ carries the same information as the marked point process. Note that the paths are càdlàg and have at most finitely many jumps in finite time. A third option to represent the insurance data is the \textbf{multivariate state occupation process}
\begin{align*}
	I=(I^i)_{i \in \mathcal{Z}}, \quad I^i: [0,\infty) \times \Omega \rightarrow \{0,1\},
\end{align*}
defined by
\begin{align*}
	I^{i}(t) := \sum_{k \in \mathbb{N}_0}\sum_{i \in \mathcal{Z}} \mathds{1}_{\{\tau_k \leq t < \tau_{k+1}\}}  \mathds{1}_{\{\zeta_k=i\}}.
\end{align*}
Again, the paths are càdlàg  and have at most finitely many jumps in finite time. The counting processes and state occupation processes satisfy the fundament relation
\begin{align}\label{RelationIN}
	I^i(t) - I^i(s) = \sum_{j:j \neq i} \int_{(s,t]} \big( N^{ji}(\d u) -  N^{ij}(\d u) \big).
\end{align}
A fourth option to  represent the insurance data is the \textbf{multi-state process}
\begin{align*}
	Z: [0,\infty) \times \Omega \rightarrow \mathcal{Z}
\end{align*}
defined by
\begin{align*}
	Z(t) := \sum_{k \in \mathbb{N}_0}\mathds{1}_{\{\tau_k \leq t < \tau_{k+1}\}}  \zeta_k = \sum_{i \in \mathcal{Z}}   i \, I^i(t) .
\end{align*}
Also this process has càdlàg paths with at most finitely many jumps in finite time. For notational convenience, we define $Z(0-) := Z(0)$.

The development of the observable information for an individual insurance contract is described by the filtration $\mathcal{F}=(\mathcal{F}_t)_{t \in [0,\infty)}$ defined by
\begin{align*}
	\mathcal{F}_t := \sigma( \mathds{1}_{\{\tau_k \leq t\}} (\tau_k,\zeta_k): k \in \mathbb{N}_0, i \in \mathcal{Z})
\end{align*}
Note that
\begin{align*}
\mathcal{F}_t =& \sigma (Z(0), N^{ij}(s): s \leq t,\,  i,j \in \mathcal{Z}, \, i \neq j)\\
	=& \sigma (I^{i}(s): s \leq t,\,  i \in \mathcal{Z})\\
	=& \sigma (Z(s): s \leq t),
\end{align*}
which means that the marked point process, the multivariate counting process together with the initial state, the multivariate occuptation process, and the multi-state process all generate the same information. Consequently, on the intere time line, all these processes are measurable with respect to the $\sigma$-algebra
\begin{align*}
\mathcal{F}_{\infty} := \sigma (\mathcal{F}_u: u \in [0,\infty) ).
\end{align*}
While the $\sigma$-algebra $\mathcal{F}_t$ describes the observable information of the present and the past,  the interval $[0,t]$, the observable information of the past only, corresponding to the interval $[0,t)$, is given by 
\begin{align*}
	\mathcal{F}_{t-}:= \sigma( \mathds{1}_{\{\tau_k < t\}} (\tau_k,\zeta_k): k \in \mathbb{N}_0, i \in \mathcal{Z})
\end{align*}
Note that $\mathcal{F}_{0-}$ is the trivial $\sigma$-algebra and, similar to before,
\begin{align*}
\mathcal{F}_{t-}
=& \sigma (Z(0) \mathds{1}_{\{t>0\}}, N^{ij}(s): s < t,\,  i,j \in \mathcal{Z}, \, i \neq j)\\
=& \sigma (I^{i}(s): s < t,\,  i \in \mathcal{Z})\\
=& \sigma (Z(s): s < t)\\
= & \sigma (\mathcal{F}_s: s< t).
\end{align*}
The family of the left-limit $\sigma$-algebras
\begin{align*}
\mathcal{F}^-:=(\mathcal{F}_{t-})_{t \in [0,\infty)}
\end{align*}
is also a filtration.

The information provided by the first $n \in \mathbb{N}_0$ elements of the marked point process  is given by the $\sigma$-algebra
\begin{align*}
	\mathcal{G}_n:= \sigma ( (\tau_k, \zeta_k): k\leq n).
\end{align*}
For any $\sigma$-algebra $\mathcal{A}$ on $\Omega$ and any  event  $B \subset \Omega$, the so-called trace $\sigma$-algebra  is defined by $\mathcal{A} \cap B:= \{A\cap B: A \in \mathcal{A}\}$.	For each   $t \in [0,\infty) $ and $n \in \mathbb{N}_0$, it holds that
	\begin{align}\begin{split}\label{GnFtEquivalence}
		\mathcal{G}_n \cap \{ \tau_n \leq t < \tau_{n+1} \} & = 		\mathcal{F}_t \cap \{ \tau_n \leq t < \tau_{n+1} \},\\
	\mathcal{G}_n \cap \{ \tau_n < t \leq \tau_{n+1} \}  &= 		\mathcal{F}_{t-} \cap \{ \tau_n < t \leq \tau_{n+1} \} .
\end{split}\end{align}
The $\sigma$-algebra $\mathcal{G}_n$ is equivalent to the stopping time $\sigma$-algebra $\mathcal{F}_{\tau_n}$:
\begin{align}\label{FtaunGn}
	\mathcal{F}_{\tau_n}:= \Big\{ A \in \mathcal{F}_{\infty}: A \cap \{ \tau_n \leq t \}  \in  \mathcal{F}_t, \,  t \geq 0 \Big\}= \mathcal{G}_n.
\end{align}
For each scenario $\omega =(t_k,z_k)_{k \in \mathbb{N}_0} \in\Omega $ and tuple $(s,i) \in [0,\infty) \times \mathcal{Z}$, we define a so-called  $(s,i)$-stopped status development $\omega_s^i$ by 
\begin{align}\label{omegastopped}
	\omega_{s}^i:= \begin{cases} ((t_0,z_0),\ldots, (t_n,z_n),  (\infty, \nabla) , \ldots) &:\, t_n < s \leq t_{n+1}, z_n=i,\\
		 ((t_0,z_0), \ldots,(t_n,z_n), (s,i),  (\infty, \nabla) , \ldots) &:\, t_n < s \leq t_{n+1}, z_n \neq i, \\
		 ((0,i), (\infty, \nabla) , \ldots) &:\, s = 0.
	\end{cases}
\end{align}
The stopped status development $\omega^{i}_s$ is still an element of $\Omega$. The $(s,i)$-stopping  leaves  the multi-state process $Z$ unchanged on $[0,s)$ and makes it constantly equal to $i$ on $[s,\infty)$.

\section{Transition rates and initial distribution}

There are several ways to specify a probability measure $\mathbb{P}$ on the measurable space $(\Omega,\mathcal{F}_{\infty})$. Our approach is to start from the representation
\begin{align*}
\mathcal{F}_{\infty} = \sigma (Z(0), N^{ij}: i,j \in \mathcal{Z}, \, i \neq j)
\end{align*}
and to specify the distribution of $Z(0)$ and $N$. We assume that we have the  \textbf{initial distribution} of $Z$, specified by a distribution function
\begin{align*}
\alpha:\mathcal{Z}\rightarrow  [0,1], \quad \sum_{i \in \mathcal{Z}} \alpha (i)=1,
\end{align*}
for the initial state $Z(0)$, and we assume that we have \textbf{(cumulative) transition rates}
\begin{align*}
\Lambda=(\Lambda^{ij})_{i,j \in \mathcal{Z}: i \neq j}, \quad \Lambda^{ij}:[0,\infty)\times \Omega \rightarrow \mathbb{R},
\end{align*}
for the multivariate counting process $N$. The total (combined) transition rate for leaving a current state $i \in \mathcal{Z}$ is in brief written as
\begin{align*}
	\Lambda^{i\largecdot} := \sum_{j:j \neq i} \Lambda^{ij}.
\end{align*}
The increments of $\Lambda^{ij}$  are meant to describe the expected increments of $N^{ij}$, conditional on the past information and provided that the last state was $i$, symbolically written as
\begin{align}\label{IntuitFormulaLambda}
	\Lambda^{ij}(\d t) = \E[ N^{ij}(\d t) | \mathcal{F}_{t-}],  \quad Z(t-)=i.
\end{align}
This intuitive equation is not mathematically precise. For a rigorous definition, we need to disentangle the information time variable from the integration variable. Based on the observation that
\begin{align*}
	\mathcal{F}_{t-}  \cap \{ s < t \leq \tau(s)\}& =  \mathcal{F}_{s}  \cap \{ s < t \leq \tau(s)\}, 
\end{align*}
for any integrable random variable $Y$ one can show that almost surely
\begin{align}\label{OldRepresentationExpectation}
	 \E[Y| \mathcal{F}_{t-}]=  \frac{\E[\mathds{1}_{\{\tau(s) \geq t\}}Y| \mathcal{F}_{s}]}{\mathbb{E}[\mathds{1}_{\{\tau(s) \geq t\}}|  \mathcal{F}_{s}]}, \quad s < t \leq \tau(s),
\end{align}
compare with Remark~4.2.3 in~\cite{Jacobsen2006}. Here the random variable $\tau(s)$ is defined as the first jump of $Z$ on $(s, \infty]$,
\begin{align*}
	\tau (s):= \sum_{n \in \mathbb{N}_0} \tau_{n+1} \mathds{1}_{\{\tau_n \leq s < \tau_{n+1}\}}.
\end{align*}
Based on~\eqref{OldRepresentationExpectation}, we rewrite the symbolical characterization~\eqref{IntuitFormulaLambda} of $\Lambda^{ij}$ to obtain the mathematically rigorous equation
\begin{align}\label{MathCharactOfCumTransitionRates}
	\Lambda^{ij}(\d t) = \frac{\E[\mathds{1}_{\{\tau(s) \geq t\}} N^{ij}(\d t )| \mathcal{F}_{s}]}{\mathbb{E}[\mathds{1}_{\{\tau(s) \geq t\}}|  \mathcal{F}_{s}]}, \quad  Z(t-)=i,\,s < t \leq \tau(s).
\end{align}
Throughout, we use for any subintervals $I, J$ of $[0,\infty)$ the shorthand notation
\begin{align*}
F (\d t) = H(t) G(\d t)  \;  \forall t \in I \quad \Longleftrightarrow \quad   \int_{J} F (\d t) = \int_J H(t) G(\d t)  \; \forall J \subset I{.}
\end{align*}
It is worthwhile to note that~\eqref{MathCharactOfCumTransitionRates} already reveals a certain ambiguity or flexibility in the choice of (cumulative) transition rates. To see this, consider
\begin{align*}
C=(C^j)_{j \in \mathcal{Z}}, \quad C^j:[0,\infty)\times \Omega \rightarrow \mathbb{R},
\end{align*}
defined by
\begin{align}\label{DefinitionCompensator}
C^j(\mathrm{d}t) := \sum_{i \in \mathcal{Z}} I^i(t-) \Lambda^{ij}(\mathrm{d}t), \quad t\geq0.
\end{align}
Then~\eqref{MathCharactOfCumTransitionRates} would still hold with $\Lambda^{ij}$ replaced by $C^j$. Later, we shall demonstrate that $C^j$ is a compensator of
\begin{align}\label{DefinitionTotalCounts}
N^j := \sum_{i:i \neq j} N^{ij}.
\end{align}
To remove this ambiguity and obtain an approach that is consistent with the specification of Markov processes through deterministic (cumulative) transition rates, we shall require that
\begin{align}\label{NoAmbioguityTransitionRates}
\Lambda^{ij}(t)(\omega) = \Lambda^{ij}(t)(\omega_{\tau_n(\omega)}^i), \quad \tau_n (\omega)  < t \leq \tau_{n+1}(\omega).
\end{align}
We therefore, all in all, make the following technical assumptions for $\Lambda$:
\begin{assumption}\label{as:tech_lambda}\leavevmode
\begin{itemize}
	 \item[(a)] {For $0 \leq t \leq \tau_1$, let $\Lambda(t)$ be deterministic. For $n\in\mathbb{N}$ and $\tau_n < t \leq \tau_{n+1}$, let $\Lambda^{ij}(t)$ be measurable with respect to
	 \begin{align*}
	 \sigma(\tau_0,\zeta_0,\ldots,\tau_{n-1},\zeta_{n-1},\mathds{1}_{\{\zeta_{n-1} \neq i\}}\tau_n), \quad i\in\mathcal{Z}.
	 \end{align*}}
	 \item[(b)] Let $\Lambda$ be right-continuous and, except in a finite number of points in every finite time interval, non-decreasing with
    \begin{align*}
    \Delta \Lambda^{i\largecdot}(t) \leq 1, \quad i \in \mathcal{Z},\, t >0.
    \end{align*}
      \item[(c)] If $\Lambda^{ij}(\cdot)(\omega)$ jumps downward at time $r>0$, denoted as \textbf{reset point}, let  
      $$ \Lambda^{ij}(r-)(\omega) = \infty, \quad \Lambda^{ij}(r)(\omega) = 0.$$
	\item[(d)] 	For any sequence $i_{0}\neq i_1 \neq \cdots \neq  i_{n}$  of states  in $ \mathcal{Z}$ with $i_0=i_n$,  let at least one of  the transition rates $\Lambda^{i_0i_1}, \ldots, \Lambda^{i_{n-1}i_{n} }$ be bounded (uniformly on $\Omega$) on finite intervals.
\end{itemize}
\end{assumption}
Assumption~\ref{as:tech_lambda}(a) {guarantees that $\Lambda$ does not use superfluous information and is an equivalent assumption} to~\eqref{NoAmbioguityTransitionRates}. The right-continuity of Assumption~\ref{as:tech_lambda}(b) follows the right-continuity convention for the counting processes, while the upper jump bound prevents the transition probabilities from becoming greater than one.  An upward jump in a cumulative transition rate corresponds to a discrete probability mass for the corresponding transition. The monotony statement in Assumption~\ref{as:tech_lambda}(b) prevents the transition probabilities from becoming negative, but we allow for downward jumps at so-called reset points, of which every path has at most a finite number in every finite time interval. Reset points are necessary when transition probabilities converge continuously to one in finite time, because such a convergence implies a pole for the cumulative transition rate, which must be reset to continue the model after the pole, confer with Assumption~\ref{as:tech_lambda}(c). The downward jump is not a probability mass and must be separated when calculating probabilities. Assumption~\ref{as:tech_lambda}(d) is a sufficient condition to exclude explosions of the counting processes. {In particular, the assumptions ensure that $\Lambda^{ij}$ and $C^j$ are $\mathcal{F}^-$-adapted.}

\section{Canonical probability model}\label{sec:canonical_prob_model}

This section shows that the initial distribution $\alpha$ and the transition rates $\Lambda=(\Lambda^{ij})_{i,j:i \neq j}$ uniquely define a probability measure $\mathbb{P}$ on the  measurable space $(\Omega,\mathcal{F}_{\infty})$.
In addition, we  aim to define conditional probability measures 
\begin{align}\label{ProbKernels}
	\mathbb{P}_{s}^i[\,\cdot\, ]=\mathbb{P}[\, \cdot \,  |  \mathcal{F}_{s-}, Z(s)=i], \quad s \in[0,\infty),\, i \in \mathcal{Z}.
\end{align}
More precisely, we are looking for $\mathcal{F}_{s-}$-measurable probability kernels $\mathbb{P}_{s}^i$, $s \in[0,\infty)$, $ i \in \mathcal{Z}$,  that satisfy almost surely for each  $A \in \mathcal{F}_{\infty}$ the equation
\begin{align} \label{DefOfProbKernel}
	I^i(s) \mathbb{P}_{s}^i[A] =  I^i(s) \mathbb{P}[ A | \mathcal{F}_s].
\end{align}
Given the probability kernels  $\mathbb{P}_{s}^i$, $s \in[0,\infty)$, $ i \in \mathcal{Z}$, we  moreover define corresponding conditional expectations by
\begin{align}\label{DefExpectIntTheProof}
	\E_{s}^i[\, \cdot \,]:= \int_{\Omega} (\cdot) \d \mathbb{P}_{s}^i.
\end{align}

\begin{theorem} \label{TheoremCanoncialSpace}
	There exist unique $\mathcal{F}_{s-}$-measurable  probability kernels $\mathbb{P}_{s}^i$, $s \in [0,\infty)$, $i\in\mathcal{Z}$,  on $(\Omega, \mathcal{F}_{\infty})$   such that 
\begin{enumerate}
    \item[(i)]  For $s \in [0,\infty)$, $\omega \in \Omega$  and $i,j \in \mathcal{Z}$ with $j \neq i$, it holds that
\begin{align*}
 \Lambda^{ij}(\d t)(\omega) = \frac{\E_s^i[\mathds{1}_{\{\tau(s) \geq t\}} N^{ij}(\d t )](\omega)}{\mathbb{E}_s^i[\mathds{1}_{\{\tau(s) \geq t\}}](\omega)} 
\end{align*}
for those $t \in (s,\infty)$ such that neither $Z(\cdot )(\omega)$ has a jump in $(s,t)$ nor $\Lambda^{i \largecdot}(\cdot)(\omega)$ has a pole or a jump of size $+1$ in $(s,t{]}$. 
\end{enumerate}
    Further, there exists a unique probability measure $\mathbb{P}$ on $(\Omega, \mathcal{F}_{\infty})$ such that
    \begin{enumerate}
\item[(ii)] Equation~\eqref{DefOfProbKernel} holds almost surely for $A \in \mathcal{F}_{\infty}$, $s \in [0,\infty)$,
	\item[(iii)] {f}or  $i \in \mathcal{Z}$, it holds that  
   \begin{align*}
   	 \alpha(i) =\mathbb{P}[Z(0) = i  ].
   \end{align*}
	\end{enumerate}
\end{theorem}
Property~(i) refers to Equations~\eqref{MathCharactOfCumTransitionRates}--\eqref{NoAmbioguityTransitionRates} and helps clarify in what sense $\Lambda$ indeed represents the transition rates with respect to the probability measure $\mathbb{P}$.
\begin{proof}
We start by explicitly constructing the probability kernels and the probability measure, and we do this on the extended measurable space  $( \widetilde{\Omega},  \widetilde{\mathcal{F}}_{\infty})$   defined by
	\begin{align*}
		\widetilde{\Omega} := \big\{ (t_l,z_l)_{l \in \mathbb{N}_0}  :  t_0=0,\,& t_l < t_{l+1} \textrm{ for } t_l< \infty, t_l=t_{l+1} \textrm{ for }t_l= \infty, \\
		& z_l \neq z_{l+1}  \textrm{ for } z_l \in \mathcal{Z} ,z_l =z_{l+1}  \textrm{ for } z_l =\nabla \big\}
	\end{align*}
	and  
	$$\widetilde{\mathcal{F}}_{\infty}:= \sigma (\tau_l, \zeta_l: l \in \mathbb{N}_0)$$ for 
	$\tau_l, \zeta_l $ defined similarly to before but on the extended domain  $\widetilde{\Omega}$. Later on we will show that the difference $\widetilde{\Omega}  \setminus \Omega$ has probability zero, which will bring us back to the original measurable space. 
For $ s \in [0,\infty)$ and $i \in \mathcal{Z}$,  the random variable
\begin{align}\label{DefRho}
	\rho^i_s(\omega):= \sup \bigg\{ u \geq s : \sup_{t \in [s,u]}  \Lambda^{i\largecdot}(t) (\omega_s^i)< \infty \bigg\},
\end{align}
gives the first reset point of  $\Lambda^{i\largecdot}(\cdot)(\omega_s^i)$ on $[s,\infty)$. 	
	Since the paths of $\Lambda_{ij}(\cdot )(\omega_s^i)$ are non-decreasing on $[s,\rho_i(\omega))$, they can have at most countably many jumps on $[s,\rho_i(\omega))$. 
Let 
$$\Lambda_c^{ij}(\d t)(\omega_s^i):= \Lambda_{ij}(\d t)(\omega_s^i) - \Delta \Lambda^{ij}( t)(\omega_s^i)$$
denote the continuous part of $\Lambda^{ij}( \cdot)(\omega_s^i)$ on $[s,\rho_i(\omega))$. In the following, let  
$$\omega =(t_l,z_l)_{l \in \mathbb{N}_0}\in \widetilde{\Omega}$$ be arbitrary but fixed.
For  $i,j \in \mathcal{Z}$, $ i \neq j$  we define mappings 
$p_s^i$ and $p_s^{ij}$ on  $ [s, \infty)$ by
\begin{align}\label{pnipnij} \begin{split}
			p^i_s(t)(\omega) &:=  e^{ -( \Lambda_c^{i\largecdot}(t \wedge \rho_s^i(\omega))(\omega_{s}^i)-\Lambda_c^{i\largecdot}(s)(\omega_{s}^i)) } \prod_{s < u \leq t \wedge \rho_s^i(\omega)} \Big(1-  \Delta\Lambda^{i\largecdot}(u)(\omega_{s}^i)\Big), \\
			p_s^{ij}(t)(\omega) &:= \int_{(s,t ] \cap (0,\rho_s^i(\omega))} p_s^i(u-)(\omega)\,\Lambda^{ij}(\d u)(\omega_{s}^i).
\end{split} \end{align}
The latter definitions imply that 
\begin{align}\label{Relationpipij}\begin{split}
		\sum_{j:j \neq i} p_s^{ij}(t) &= \int_{(s,t]\cap (0,\rho_s^i(\omega))} p_s^i(u-)(\omega) \big( - \Lambda^{i\largecdot}(\d u)(\omega_{s}^i) \big)\\
		&=-\int_{(s,t]\cap (0,\rho_s^i(\omega))} p_s^i (\d u)(\omega)\\
		&=\begin{cases}1- p_s^i(t) &:\, t \in [s,  \rho_s^i(\omega)), \\ 0 &: t \in [\rho_s^i(\omega),\infty).\end{cases}
\end{split}\end{align}
 We extend the domains  of $p_s^i(\cdot)(\omega)$ and $p^{ij}_s(\cdot)(\omega)$  from $[s,\infty)$ to $[s,\infty]$ by setting
\begin{align*} 
		p_s^i(\infty)&:=p_s^i(\infty-),\\
		p_s^{ij}(\infty)&:=p_s^{ij}(\infty-), \quad i,j \in \mathcal{Z}, \, i \neq j.
\end{align*}
Furthermore, we define
$$ p_s^{i \nabla}(t):= \mathds{1}_{\{\infty\}}(t), \quad  t \in [s,\infty].$$
With these extensions, the  mapping   $(t,j) \mapsto p_s^{ij}(t)$ defines a conditional probability distribution on $[s, \infty] \times  (\mathcal{Z} \cup \{\nabla\})$ for each $(s,i)$. For $(s,i)=(t_n,z_n)$, we interpret this conditional probability distribution as the probability kernel
\begin{align}\label{Ptaun+1Fs-Zsi}
p_{\tau_n}^{\zeta_n j}(t)(\omega)= \mathbb{P}[\tau_{n+1} \leq t, \zeta_{n+1}=j|  (\tau_l, \zeta_l)_{l \leq n}](\omega), \quad t \geq \tau_n(\omega).
\end{align}
However, we still have to show that the mapping $(s,i,\omega) \mapsto p_s^{ij}(t)(\omega)$ is measurable for each $(t,j)$. 
 The mapping $(s,i,\omega) \mapsto  (s,i,\omega_s^i)$ is measurable as a mapping from the measurable space $([0,\infty) \times \mathcal{Z}\times \Omega, \mathcal{B}([0,\infty))\otimes 2^{\mathcal{Z}}  \otimes \mathcal{F}_{\infty})$ to the same  space, since it is  composed of simple functions and countably many case differentiations, see \eqref{omegastopped}. As the transition rates are right-continuous by definition, they are jointly measurable, so  the composition $(s,i,\omega) \mapsto \Lambda^{ij}(s)(\omega_s^i)$ is measurable too. Likewise, one can argue that   $( s,i,\omega,t) \mapsto \Lambda^{ij}(t)(\omega_s^i)$ is measurable, and the arguments still apply for the continuous part $\Lambda^{ij}_c$ and the pure jump part $\Lambda^{ij}-\Lambda^{ij}_c$ of $\Lambda^{ij}$. From all these measurable mappings,  the mapping $(s,i, \omega,t) \mapsto p^i_s(t)(\omega)$ and then $(s,i,\omega,t) \mapsto p^{ij}_s(t)(\omega)$ are formed by using simple operations and limits, see \eqref{pnipnij}, so they are measurable too. That means that \eqref{Ptaun+1Fs-Zsi} indeed describes probability kernels.

By applying the Ionescu-Tulcea theorem, confer with Proposition~V.1.1 in~\cite{Neveu1965}, from the distribution function  $(t_0,z_0) \mapsto \mathds{1}_{\{0\}}(t_0)  \alpha(z_0)  $ for $ (\tau_{0}, \zeta_{0})$ and the   probability kernels  \eqref{Ptaun+1Fs-Zsi} we construct a probability measure $\mathbb{P}$ on  $( \widetilde{\Omega},  \widetilde{\mathcal{F}}_{\infty})$ as the unique completion of 
\begin{align} \label{DefOfP} \begin{split}
		&\mathbb{P} [ A] := \sum_{i_0, \ldots, i_m}\int_{(s_0,\infty]} \cdots \int_{(s_{l-1},\infty]} \mathds{1}_A(\omega_{s_0\cdots s_{l}}^{i_0 \cdots i_{l}}) \, p_{s_{l-1}}^{i_{l-1}i_l}(\d s_{l})(\omega_{s_0\cdots s_{l-1}}^{i_0 \cdots i_{l-1}}) \cdots  p_{s_0}^{i_0i_1}(\d s_{1})(\omega_{s_0}^{i_0})\,  \alpha(i_0),
\end{split}\end{align}
for $s_0:=0$, $A\in \mathcal{G}_l$, $l \in \mathbb{N}$,  and for  $\omega_{s_0\cdots s_l}^{i_0 \cdots i_l}$ defined by the iterative formula 
$$\omega_{s_0\cdots s_l}^{i_0 \cdots i_l} := (\omega_{s_0 \cdots s_{l-1}}^{i_0 \cdots i_{l-1}})_{s_l}^{i_l}, \quad l =1, \ldots, m.$$
Equation \eqref{DefOfP}  implies property (iii),  but  for the extended probability space  $( \widetilde{\Omega},  \widetilde{\mathcal{F}}_{\infty})$.
Analogously, we define probability kernels $\mathbb{P}_{s}^i$, $i\in\mathcal{Z}$, $s \in [0,\infty)$, on $( \widetilde{\Omega},  \widetilde{\mathcal{F}}_{\infty})$ as the unique completions of
 \begin{align}\label{DefOfcP}  \begin{split}
	\mathbb{P}_{s_n}^{i_n}[  A](\omega) :=    \sum_{i_{n+1}, \ldots, i_l}  \int_{(s_n,\infty]} \cdots  \int_{(s_{l-1},\infty]} \mathds{1}_A(\omega_{s_n\cdots s_{l}}^{i_n \cdots i_{l}}) \, p_{s_{l-1}}^{i_{l-1}i_l}(\d s_{l})(\omega_{s_n\cdots s_{l-1}}^{i_n \cdots i_{l-1}})  \cdots p_{s_n}^{i_ni_{n+1}}(\d s_{n+1})(\omega_{s_n}^{i_n})
\end{split}\end{align}
for $t_n \leq s_n < t_{n+1}$, $A\in \mathcal{G}_l$, $n,l \in \mathbb{N}_0$, $n<l$.
For $n< k \leq l$, the term $\mathbb{P}_{s_{k}}^{i_{k}}[A](\omega_{s_n\cdots s_{k-1}}^{i_n \cdots i_{k-1}}) $ equals the $l-k$ inner integrals of  $\mathbb{P}_{s_n}^{i_n}[A](\omega)$, so it holds that
\begin{align}\label{ProjectionProperty2}
	\mathbb{P}_{s_n}^{i_n}[ A ]= \mathbb{E}_{s_n}^{i_n}\big[   \mathbb{P}_{\tau_{k}}^{\zeta_{k}}[A] \big], \quad k> n,\,  \tau_n \leq s_n < \tau_{n+1}.
\end{align}
Moreover, by using the fact that  definition \eqref{pnipnij} implies that 
\begin{align*}
   p_{s_n}^{i_ni_{n+1}}(\d s_{n+1})(\omega_{s_n}^{i_n}) = 	 p_{s_n}^{i_n}(s)(\omega_{s_n}^{i_n}) \, p_{s}^{i_ni_{n+1}}(\d s_{n+1})(\omega_{s_n}^{i_n}),\quad s_n \leq s < s_{n+1},
\end{align*}
from definition \eqref{DefOfcP}, equation \eqref{Relationpipij}, and the property  $(\omega_{s_n}^i)_{s}^i= \omega_{s_n}^i$ for $s \geq s_n $, we can conclude that
\begin{align*}
	&\mathbb{P}_{s_{n}}^{i_n} [A   \cap \{\tau_n \leq s < \tau_{n+1}\}](\omega)\\ &=  \mathbb{P}_{s}^{i_n} [A\cap \{\tau_n \leq s < \tau_{n+1}\}](\omega_{s_n}^{i_n}) \,p_{s_n}^{i_n}(s)(\omega_{s_n}^{i_n})\\
	& = \sum_{i_{n+1}}\int_{(s,\infty]}  \mathbb{P}_{s}^{Z(s)(\omega_{s_n}^{i_n})} [A\cap  \{\tau_n \leq s < \tau_{n+1}\}](\omega_{s_n}^{i_n})\, p_{s_n}^{i_ni_{n+1}}(\d s_{n+1})(\omega_{s_n}^{i_n})\\
	&=   \mathbb{E}_{s_n}^{i_n}\big[ \mathbb{P}_{s}^{Z(s)} [A\cap  \{\tau_n \leq s < \tau_{n+1}\}] ](\omega), \quad t_n \leq s_n < t_{n+1}, s \geq s_n.
\end{align*}
This fact and equation \eqref{ProjectionProperty2} yield
\begin{align}\label{ProjectionProperty3}\begin{split}
	\mathbb{P}_{u}^{i} [A ]
	&= \sum_{n=0}^{\infty}\mathds{1}_{\{\tau_k \leq u < \tau_{k+1}\}}\sum_{k=n}^{\infty} \mathbb{P}_{u}^{i} [ A\cap \{\tau_k \leq s < \tau_{k+1}\}]\\
	& = \sum_{n=0}^{\infty}\mathds{1}_{\{\tau_k \leq u < \tau_{k+1}\}}\sum_{k=n}^{\infty} \mathbb{E}_{u}^{i} [ \mathbb{E}_{\tau_k}^{\zeta_k} [ \mathbb{P}_s^{Z(s)}[ A\cap \{\tau_k \leq s < \tau_{k+1}\}]]]\\
	& = \mathbb{E}_{u}^{i} [  \mathbb{P}_s^{Z(s)}[ A]], \quad   s \geq u.
\end{split}\end{align}
Furthermore, since $I^j(s)((\cdot)_s^{i} ) =\mathds{1}_{j =i} $ and  $ \mathds{1}_C((\cdot)^i_{s})=\mathds{1}_C(\cdot)$, 
definition \eqref{DefOfcP}  implies that  
\begin{align}\label{ExactnessOfCondExp}
	\mathbb{P}_{s}^i[ A \cap C \cap \{Z(s)=j\}] = \mathds{1}_{j =i}\, \mathds{1}_C \, \mathbb{P}_{s}^i[  A], \quad C \in \mathcal{F}_{s-}.
\end{align}
By applying \eqref{ProjectionProperty2}, \eqref{ProjectionProperty3}, and \eqref{ExactnessOfCondExp}, we can show that
\begin{align*}
	\mathbb{P} [A \cap C \cap \{Z(s)=i\}]=
	\mathbb{E}\big[  \mathds{1}_C\,I_i(s)\,\mathbb{P}_{s}^i [A] \big], \quad C \in \mathcal{F}_{s-},
\end{align*}
which is property (ii),  but for the extended probability space  $ ( \widetilde{\Omega},  \widetilde{\mathcal{F}}_{\infty})$. 
Definition \eqref{DefOfcP} implies that 
\begin{align*}
	 \mathbb{P}_{s}^i[ \tau_{n+1} \leq t , \zeta_{n+1}=j ] = p_s^{ij}(t), \quad \tau_n \leq s < \tau_{n+1}.
\end{align*}
Because of \eqref{ExactnessOfCondExp} and 
$$\{\tau_n \leq s < \tau_{n+1}\} \cap \{\tau_{n+1} \leq t, \zeta_{n+1}=j\}= \{\tau_n \leq s < \tau_{n+1}\}\cap \{\tau(s)  \leq t, Z(\tau(s))=j\},$$
we moreover have
\begin{align*}
	\mathbb{P}_{s}^i[ N^{ij}(t \wedge \tau(s) ) -N^{ij}(s)] &= p_s^{ij}(t), \quad s < t  < \rho_s^i,
\end{align*} 
and by equation \eqref{Relationpipij} we furthermore get
\begin{align*}
	\mathbb{P}_{s}^i[ \tau(s) \geq t ] &= p_s^{i}(t-), \quad s  < t < \rho_s^i.
\end{align*} 
The latter two equations, definition \eqref{pnipnij}, and the measurability assumption \eqref{NoAmbioguityTransitionRates}  yield property (i). 

In a next step, we are going to show that our construction still works on the restricted  measurable space  $(\Omega,  \mathcal{F}_{\infty}) \subset ( \widetilde{\Omega},  \widetilde{\mathcal{F}}_{\infty})$.
For  $r \geq  s$, the definitions \eqref{DefOfcP} and \eqref{pnipnij} 
imply that
\begin{align*} 
	&\mathbb{E}_{s}^{i}\bigg[\int_{(\tau_n, \tau_{n+1}]}  \mathds{1}_{(r,t]}  N^{jk}( \d u )\bigg](\omega)\\
	&=\mathbb{P}_{s}^{i}[ r <  \tau_{n+1} \leq t, \zeta_{n+1}=k, \zeta_n=j](\omega)\\
	& =   \int_{(s, \infty]}\mathds{1}_{(r,t]} \mathds{1}_{i=j} \, p_{s}^{ik}(\d u)( \omega_{s}^i)\\
	& =   \int_{(s, \infty]}  \mathds{1}_{(r,t]} (\omega)\mathds{1}_{i=j} \, p_{s} ^{i}(u-)(\omega) \, \Lambda^{jk}(\d u )(\omega_{s}^i)\\
	& =    \int_{(s, \infty]} \mathds{1}_{(r,t]} \,\mathbb{P}_{s}^i [Z(u-)=j, \tau_{n+1}  \geq u ](\omega)\, \Lambda^{jk}(\d u )(\omega_{s}^i), \quad t_n \leq s < t_{n+1}.
\end{align*}
On the other hand, by applying definition \eqref{DefOfcP}, using the fact that 
$$I^j(u-)(\omega )\Lambda^{jk}(\d u)(\omega) =I^j(u-)(\omega_{t_n}^i) \Lambda^{jk}(\d u)(\omega_{t_n}^i), \quad  u \leq t_{n+1},\, z_n=i,$$ and  applying Tonelli's theorem, we can show that 
\begin{align*}
	&   \E_{s}^i\bigg[   \int_{(\tau_n  , \tau_{n+1}]}  \mathds{1}_{( r,t]}  I^j(u-) \,  \Lambda^{jk}(\d u )\bigg](\omega)\\
	&=  \int_{(s, \infty]}   \bigg(\int_{(s , s_{n+1}]}   \mathds{1}_{(r,t]} (\omega_{s}^i) I^j(u-)(\omega_{s}^i)\,  \Lambda^{jk}(\d u )(\omega_{{s}}^{i}) \bigg)\mathbb{P}_{s}^i[\tau_{n+1} \in \d s_{n+1}](\omega)\\
	&=   \int_{ (s, \infty]}  \mathds{1}_{(r,t]} (\omega) \int_{[u, \infty]}  I^j(u-)(\omega_{s}^i) \,\mathbb{P}_{s}^i[\tau_{n+1} \in \d s_{n+1}](\omega) \,  \Lambda^{jk}(\d u )(\omega_{{s}}^{i}) \\	
	& =    \int_{ (s,\infty]}  \mathds{1}_{(r,t]}  \,\mathbb{P}_{{s}}^i [Z(u-)=j,  \tau_{n+1} \geq u ](\omega)\, \Lambda^{jk}(\d u )(\omega_{{s}}^i), \quad t_n \leq s < t_{n+1},
\end{align*}
so that we can conclude that 
\begin{align*}
	& \mathbb{E}_{s}^i\bigg[ \int_{(r,t]} \mathds{1}_{\{\tau_n < u \leq \tau_{n+1}\}} N^{jk}(\d u) \bigg]=  \mathbb{E}_{s}^i\bigg[ \int_{(r,t]} \mathds{1}_{\{\tau_n < u \leq \tau_{n+1}\}} I^j(u-) \,  \Lambda^{jk}(\d u ) \bigg], \quad \tau_n \leq s < \tau_{n+1}, \, r \geq s.
\end{align*}
By setting $(s,i)=(\tau_n,\zeta_n)$, using equation \eqref{ExactnessOfCondExp} for pulling the factor $\mathds{1}_{r \geq \tau_n }$ inside the conditional expectations, and  applying equation \eqref{ProjectionProperty3} together with Tonelli's theorem, we get
\begin{align*}
	&  \mathbb{E}_{v}^l\bigg[ \int_{(r\vee v,t]} \mathds{1}_{\{\tau_n < u \leq \tau_{n+1}\}} N^{jk}(\d u)  \bigg]=   \mathbb{E}_{v}^l\bigg[ \int_{(r\vee v,t]} \mathds{1}_{\{\tau_n < u \leq \tau_{n+1}\}}  I^j(u-) \,  \Lambda^{jk}(\d u ) \bigg].
\end{align*}
By applying  the latter two equations and using equation \eqref{ExactnessOfCondExp} for pulling the factor $\mathds{1}_{ \tau_n \geq v }$  out of the conditional expectations, for any $(v,l) \in [0,\infty) \times \mathcal{Z}$ with $v \leq r$ we can conclude that
\begin{align} \label{CompensatorEquation}\begin{split}
	\mathbb{E}^l_v\Big[ N^{jk}(t \vee v )- N^{jk}(r \vee v)\Big] & = \sum_{n \in \mathbb{N}_0} 	\mathbb{E}^l_v\bigg[ \int_{(r\vee v,t]} \mathds{1}_{\{\tau_n < u \leq \tau_{n+1}\}} N^{jk}(\d u)  \bigg]\\
	&=  \mathbb{E}_v^l\bigg[  \int_{(r\vee v,t\vee v]} I_j(u-)  \Lambda^{jk}(\d u )\bigg].
\end{split}\end{align}		
Let $J \subset \{ (j,k) \in \mathcal{Z}^2: j\neq k\}$ be the set of transitions for which the corresponding transition rates are bounded on finite intervals. We assumed that for any recurrent sequence of transitions $i_0 \neq i_1 \neq \cdots \neq i_{n}$,  at least one of  the transitions $(i_{l-1},i_{l})$ is from the set  $J$. If the  sequence has a  length greater than $|\mathcal{Z}|$, then it necessarily contains a recurrent sub-sequence of maximum length $|\mathcal{Z}|$, and this sub-sequences  contains at least one transition from $J$ by our assumption. By removing the sub-sequence and iterating our arguments, we can conclude that a sequence of length $n$ contains at least $ [ n/ | \mathcal{Z}| ] +  | \mathcal{Z}|-1$ transitions from $J$. So we have
\begin{align}\label{UpperEstimateForJumps}
	\sum_{j,k: j\neq k} N^{jk}(t)& \leq  | \mathcal{Z}|- 1 + | \mathcal{Z}| \, \sum_{(j,k) \in J}N^{jk}(t)
\end{align}
for each $ t \in [0,\infty)$. 
This fact and \eqref{CompensatorEquation}  yield
\begin{align}\label{BoundForN}\begin{split}
	\mathbb{E}\bigg[ \sum_{j,k: j\neq k} N^{jk}(t)\bigg]  & = \sum_i \alpha(i) \, \mathbb{E}_0^i \bigg[ \sum_{j,k: j\neq k} N^{jk}(t)\bigg]\\
	  &\leq  \sum_i \alpha(i) \, \bigg( |\mathcal{Z}|-1+ | \mathcal{Z}|\,	\mathbb{E}\bigg[ \sum_{(j,k) \in J}\big(\Lambda^{jk}(t)-\Lambda^{jk}(0)\big)\Big] \bigg),
\end{split}\end{align}
which is finite. Therefore,  the event
$	\{ \sum_{i,j: i\neq j} N^{ij}(t)= \infty\} $ must have a probability of zero, so that
\begin{align*} 
	\mathbb{P}[\Omega]= \mathbb{P}[\lim_{n \rightarrow  \infty}\tau_n = \infty ]=1.
\end{align*}
Similarly, using \eqref{UpperEstimateForJumps} and \eqref{CompensatorEquation} one can also show that 
$\mathbb{P}_s^i[\Omega](\omega)=1$ for all $\omega \in \Omega$.  That means that $\mathbb{P}$ is a probability measure and $\mathbb{P}_s^i$, $s \geq 0$, $i \in \mathcal{Z}$, are probability kernels  also on  the restricted measurable space 
$$(\Omega, \widetilde{\mathcal{F}}_{\infty} \cap \Omega)= (\Omega, \mathcal{F}_{\infty}).$$  The properties (i) to (iii) still hold on this restricted space.

We now show uniqueness of $\mathbb{P}$ and $\mathbb{P}_{s}^i$, $s \in [0,\infty)$, $ i \in \mathcal{Z}$.  Suppose that $\widetilde{\mathbb{P}}$ and $\widetilde{\mathbb{P}}_{s}^i$, $s \in [0,\infty)$, $ i \in \mathcal{Z}$, also satisfy the properties (i) to  (iii). 
Property (i)  implies that
\begin{align*}
\widetilde{\mathbb{P}}_{s}^i[  \tau(s) > t ](\omega) & = \sum_{j:j \neq i} \int_{(t , \infty] } \widetilde{\mathbb{E}}_{s}^i[ \mathds{1}_{\{\tau(s)\geq u\}} N^{ij}( \d u ) ](\omega) \\
& = \sum_{j:j \neq i} \int_{(t , \infty] } \widetilde{\mathbb{P}}_{s}^i[  \tau(s) \geq u ](\omega) \, \Lambda^{ij}(\d u)(\omega_{s}^i) , \quad s < t <  \rho_s^i(\omega). 
\end{align*}
Since $\widetilde{\mathbb{P}}_{s}^i[  \tau(s) > s ]= \widetilde{\mathbb{P}}_{s}^i[  \Omega ](\omega)=1$, see the definition of $\tau(s)$, we have that the latter equation is equivalent to 
\begin{align*}
	 \widetilde{\mathbb{P}}_{s}^i[  \tau(s) > t ](\omega)
	& = 1- \sum_{j:j \neq i} \int_{(s,t] }  \widetilde{\mathbb{P}}_{s}^i[  \tau(s) \geq u ](\omega)  \, \Lambda^{ij}(\d u)(\omega_{s}^i) , \quad s < t <  \rho_s^i(\omega). 
\end{align*}
So, for each fixed $i \in \mathcal{Z}$, $s \in [0,\infty)$ and $\omega \in \Omega$, the function  $f(t) =  \widetilde{\mathbb{P}}_{s}^i[  \tau(s) > t ](\omega)$ solves   the Volterra integral equation
\begin{align*}
	f(t) 
	& = 1-  \int_{(s,t] } f(u-) \, \Lambda^{i\largecdot}(\d u)(\omega_{s}^i) , \quad s < t <  \rho_s^i(\omega), 
\end{align*}
which has  the product integral 
$$ f(t)=  \prod_s^t \Big(1+ \Lambda^{i\largecdot}(\d u)(\omega_{s}^i)\Big)=e^{ -( \Lambda_c^{i\largecdot}(t)(\omega_{s}^i)-\Lambda_c^{i\largecdot}(s)(\omega_{s}^i)) } \prod_{s < u \leq t} \Big(1-  \Delta\Lambda^{i\largecdot}(u)(\omega_{s}^i)\Big) , \quad s < t <  \rho_s^i(\omega), $$
as its unique solution.
That means that property (i) defines $\widetilde{\mathbb{P}}_{s}^i[  \tau(s) > t ](\omega)$  uniquely on $(s,\rho_s^i(\omega))$.  Moreover, property (i) also implies that
\begin{align*}
 \widetilde{\mathbb{P}}_{s}^i[  N^{ij}(t \wedge \tau(s)) - N^{ij}(s) ](\omega) 
& = \int_{(t , \infty] } \widetilde{\mathbb{P}}_{s}^i[  \tau(s) \geq u ](\omega) \, \Lambda^{ij}(\d u)(\omega_{s}^i) , \quad  s < t <  \rho_s^i(\omega),
\end{align*}
for $j \in\mathcal{Z}$, $j \neq i$, so that 
 also the mappings $ \widetilde{\mathbb{P}}_{s}^i[  N^{ij}(t \wedge \tau(s)) - N^{ij}(s) ](\omega)  $, $j \in \mathcal{Z}$, are unique on $(s,\rho_s^i(\omega))$. That means that the probability kernels \eqref{Ptaun+1Fs-Zsi} are uniquely characterized by  property (i), and by the Ionescu-Tulcea theorem we get that   $\widetilde{\mathbb{P}}_{s}^i=\mathbb{P}_{s}^i$, $s \in [0,\infty)$, $ i \in \mathcal{Z}$. In particular, we have  $\widetilde{\mathbb{P}}_{0}^i=\mathbb{P}_{0}^i$, so by properties (ii) and (iii) we finally get 
 $  \widetilde{\mathbb{P}}=\mathbb{P}$.
\end{proof}
The following result confirms the role of $C^j$ from~\eqref{DefinitionCompensator} as the compensator of $N^j$ defined in~\eqref{DefinitionTotalCounts}.
\begin{proposition}\label{PropMartingaleProperty}
	Let $Y$ be a jointly measurable and $\mathcal{F}^-$-adapted process such that 
	\begin{align*}
		\mathbb{E}_{s}^i\bigg[\int_{(s,t]} |Y(u)|  I^j(u-) \Lambda^{jk} (\d u)\bigg] < \infty.
	\end{align*}
	Then we have
	\begin{align*}
		\mathbb{E}_{s}^i\bigg[\int_{(s,t]} Y(u) N^{jk} (\d u)\bigg]=\mathbb{E}^i_{s}\bigg[ \int_{(s,t ]} Y(u) I^j(u-) \Lambda^{jk} (\d u)\bigg]
	\end{align*}
	for $0 \leq s < t < \infty$ and $ i,j,k \in \mathcal{Z}$ with $j \neq k$. 
\end{proposition}
\begin{proof}At first, we show the assertion for each of the bounded processes $Y_n(t):= (Y(t) \wedge n) \vee (-n)$, $n \in \mathbb{N}$. For any  $Y_n$, the proof is similar to the proof of Equation~\eqref{CompensatorEquation}, but instead of  Tonelli's theorem we apply Fubini's theorem and additionally exploit the fact that, whenever $Z(s) = i$, it holds that $Y(u)(\omega)= Y(u)(\omega_s^i)$ for $ s < u \leq \tau(s)(\omega)$.  Based on the  integrability condition, for $n$ going to infinity  we obtain the assertion also for the limit $Y= \lim_{n \rightarrow \infty} Y_n$. 
\end{proof}
\begin{definition}
Given the canonical measurable space $(\Omega,\mathcal{F}_\infty)$, the probability mesaure $\mathbb{P}$ on $(\Omega,\mathcal{F}_\infty)$, the probability kernels $\mathbb{P}_s^i$, $s\in[0,\infty)$, $i\in\mathcal{Z}$, on $(\Omega,\mathcal{F}_\infty)$, and the collection
\begin{align*}
\mathcal{P} := \big((\mathbb{P}_s^i)_{s \in [0,\infty),i\in\mathcal{Z}},\mathbb{P}\big)
\end{align*}
defined by Theorem~\ref{TheoremCanoncialSpace} are called the \emph{canonical probability measure}, \emph{canonical probability kernels}, and \emph{canonical probability model} generated by $(\alpha, \Lambda)$, respectively.
\end{definition}
From now on we generally work with the canonical probability model, since it provides the conditional expectations $\mathbb{E}_s^i[\cdot ]$, $s \in [0,\infty)$, $i \in \mathcal{Z}$, which have all the usual properties of conditional expectations not only almost surely, but everywhere on $\Omega$. 
\begin{proposition} \label{Proposition:CanoncialSpaceProperties}
	For  $s \in [0,\infty)$ and $i \in \mathcal{Z}$, let $Y$ be  a $\mathbb{P}_s^i$-integrable random variable. Then
	\begin{itemize}
		\item[(i)] For  $C \in \mathcal{F}_{s-}$ and $j \in \mathcal{Z}$  it holds that
		\begin{align*}
			\mathbb{E}_{s}^i[  \mathds{1}_C \, Y  ] &=  \mathds{1}_{C}\,  \mathbb{E}_{s}^i[ Y],\\
			\mathbb{E}_{s}^i[ I^j(s) \,  Y  ] &= \mathds{1}_{j=i} \,  \mathbb{E}_{s}^i[ Y]{,}
		\end{align*}
		\item[(ii)] {f}or $t \in [s,\infty)$ and $n \in \mathbb{N}_0$ it holds that
		\begin{align*}
			\E_{s}^i[Y] &= \E_{s}^i[\E_{t}^{Z(t)}[Y]],\\
			\mathds{1}_{\{\tau_n > s\}}	\, \E_{s}^i[Y] &= \mathds{1}_{\{\tau_n > s\}}	\,  \E_{s}^i[\E_{\tau_n}^{\zeta_n}[Y]].
		\end{align*}
	\end{itemize}
\end{proposition}
Note that this proposition genuinely does not need the usual `almost sure' constraints.
\begin{proof}
	For $X= \mathds{1}_A$, $A \in \mathcal{F}_{\infty}$, the properties (i) and (ii) have already been shown in the proof of Theorem~\ref{TheoremCanoncialSpace}, see Equations~\eqref{ExactnessOfCondExp},~\eqref{ProjectionProperty3}, and~\eqref{ProjectionProperty2}. From the building blocks $X= \mathds{1}_A$, $A \in \mathcal{F}_{\infty}$, we can construct general random variables as limits of linear combinations of these building blocks. The linearity of the expectation operator $\E_s^i[\cdot]$ and the dominated convergence theorem  then give the assertions.
\end{proof}
\begin{proposition}\label{prop:Markov_LambdaDeterministic}\leavevmode
\begin{itemize}
\item[(i)] If $\Lambda$ is deterministic, then for each  $s \in [0,\infty)$ and $i\in\mathcal{Z}$ the canonical probability kernel $\mathbb{P}_{s}^i$ restricted to $\sigma( Z(t): t\geq s)$  is  deterministic, and $Z$ is a Markov process (with respect to $\mathbb{P}$).
\item[(ii)] If $Z$ is a Markov process (with respect to $\mathbb{P}$), then there exist deterministic transition rates $\tilde{\Lambda}$ such that the canonical probability measure $\tilde{\mathbb{P}}$ generated by $(\alpha,\tilde{\Lambda})$ satisfies $\tilde{\mathbb{P}} = \mathbb{P}$.
\end{itemize}
\end{proposition}
\begin{proof}
Suppose that $\Lambda$ is deterministic. In the construction of $\mathbb{P}_s^i[A]$ in the proof of Theorem \ref{TheoremCanoncialSpace}, the value of  $ \mathbb{P}_{s}^i[A](\omega)$ depends on $\omega$ via $ \Lambda(\cdot)(\omega_s^i)$ and $\mathds{1}_A(\omega_s^i)$, see \eqref{DefOfcP}. So for $A\in \sigma( Z(s): s\geq t)$, the mapping  $ \omega \mapsto \mathbb{P}_{s}^i[A](\omega)$ is constant if $\omega \mapsto \Lambda(\cdot)(\omega)$ is constant, which means it is deterministic.   Moreover, 
the latter fact and Equation~\ref{DefOfProbKernel} imply that  $Z$ satisfies the Markov property.

Suppose that $Z$ is a Markov process. Then  for each $A \in \mathcal{F}_s$, $s \in [0,\infty)$ and $i \in \mathcal{Z}$ we almost surely have
$$  I^i(s) \mathbb{P}_{s}^i[A] = I^i(s)\mathbb{P}[A| \{Z(s)=i\}],$$
see Equation~\eqref{DefOfProbKernel}. 
By the dominated convergence theorem,  for each  $s \in [0,\infty)$ and $i \in \mathcal{Z}$, the processes  $t \mapsto \E_s^i[\mathds{1}_{\{\tau(s) \geq t\}} N^{ij}(\d t )]$ and $t \mapsto \E[\mathds{1}_{\{\tau(s) \geq t\}} N^{ij}(\d t )| \{Z(s)=i)\}]$ are right-continuous and the processes  $t \mapsto \E_s^i[\mathds{1}_{\{\tau(s) \geq t\}} ]$ and $t \mapsto \E[\mathds{1}_{\{\tau(s) \geq t\}}| \{Z(s)=i)\}]$ are left-continuous. Therefore, we can conclude that simultaneously for all $t \in [0,\infty)$, $s \in [0,\infty) \cap  \mathbb{Q}$, $i \in \mathbb{Z}$ we almost surely have
\begin{align*}   
	I^i(s)\frac{\E_s^i[\mathds{1}_{\{\tau(s) \geq t\}} N^{ij}(\d t )]}{\mathbb{E}_s^i[\mathds{1}_{\{\tau(s) \geq t\}}]}
	&=  I^i(s)\frac{\E[\mathds{1}_{\{\tau(s) \geq t\}} N^{ij}(\d t )| \{Z(s)=i\}]}{\mathbb{E}[\mathds{1}_{\{\tau(s) \geq t\}}| \{Z(s)=i\}]}\\
 & =I^i(s) \frac{\E[\mathds{1}_{\{\tau(s) \geq t\}} I^i(s )N^{ij}(\d t )]}{\mathbb{E}[\mathds{1}_{\{\tau(s) \geq t\}} I^i(s)]}\\
  & = I^i(s)\frac{\E[\mathds{1}_{\{\tau(s) \geq t\}} I^i(t- )N^{ij}(\d t )]}{\mathbb{E}[\mathds{1}_{\{\tau(s) \geq t\}} I^i(t-)]}.
\end{align*}
From this equation combined  with Theorem \ref{TheoremCanoncialSpace}(i)  we obtain
\begin{align*}
		 \mathds{1}_{\{\Lambda^{i \largecdot}(t-)< \infty\}}I^i(s) \Lambda^{ij}(\d t) =   \mathds{1}_{\{\Lambda^{i \largecdot}(t-)< \infty\}}I^i(s)  \frac{\E[\mathds{1}_{\{\tau(s) \geq t\}} N^{ij}(\d t )]}{\mathbb{E}[\mathds{1}_{\{\tau(s) \geq t\}} I^i(t-)] }, \quad s < t \leq \tau(s), 
\end{align*}
almost surely for all $s \in [0,\infty) \cap  \mathbb{Q}$, $i \in \mathbb{Z}$.
 As a consequence, for any rational partition $\mathcal{T}$ of $[0,\infty)$ we have
\begin{align*}
		&\int_{[0,u] \cap (s, \tau(s)]} \mathds{1}_{\{\Lambda^{i \largecdot}(t-)< \infty\}} I^i(s) \Lambda^{ij}(\d t)\\
		 &=  \sum_{\mathcal{T}} \int_{[0,u] \cap (s, \tau(s)] \cap (t_k,t_{k+1}]  }  \mathds{1}_{\{\Lambda^{i \largecdot}(t-)< \infty\}}  I^i(s) \frac{\E[\mathds{1}_{\{\tau(s \vee t_k) \geq t\}} N^{ij}(\d t )]}{\mathbb{E}[\mathds{1}_{\{\tau(s \vee t_k) \geq t\}} I^i(t-)] }
\end{align*}
almost surely for all $u >0$. Let $(\mathcal{T}_m)_{m \in \mathbb{N}}$ be a sequence of rational partitions of $[0,\infty)$ with vanishing maximum step length. Then the latter line has an upper bound of 
\begin{align*}
	&  \lim_{m \rightarrow \infty }\sum_{\mathcal{T}_m} \int_{[0,u] \cap (s, \tau(s)] \cap (t_k,t_{k+1}]  }  \mathds{1}_{\{\Lambda^{i \largecdot}(t-)< \infty\}}  I^i(s) \frac{\E[ N^{ij}(\d t )]}{\mathbb{E}[\mathds{1}_{\{\tau(s \vee t_k) \geq t\}} I^i(t-)] }\\
	& = \int_{[0,u] \cap (s, \tau(s)]  }  \mathds{1}_{\{\Lambda^{i \largecdot}(t-)< \infty\}}  I^i(s) \frac{\E[ N^{ij}(\d t )]}{\mathbb{E}[I^i(t-)] }
\end{align*}
since $\mathds{1}_{\{\tau(s \vee t_k) \geq t\}} N^{ij}(\d t ) \leq N^{ij}(\d t ) $ and $\{\tau(s \vee t_k) \geq t\} \uparrow \Omega$ and by monotone convergence, and  a lower bound of
\begin{align*}
  &\lim_{m \rightarrow \infty } \sum_{\mathcal{T}_m} \int_{[0,u] \cap (s, \tau(s)] \cap (t_k,t_{k+1}]  }  \mathds{1}_{\{\Lambda^{i \largecdot}(t-)< \infty\}}  I^i(s) \frac{\E[\mathds{1}_{\{\tau(s \vee t_k) \geq t\}} N^{ij}(\d t )]}{\mathbb{E}[ I^i(t-)] }\\
  & = \int_{[0,u] \cap (s, \tau(s)]  }  \mathds{1}_{\{\Lambda^{i \largecdot}(t-)< \infty\}}  I^i(s) \frac{\E[ N^{ij}(\d t )]}{\mathbb{E}[I^i(t-)] }
\end{align*}
since $ \mathds{1}_{\{\tau(s \vee t_k) \geq t\}} N^{ij}(\d t ) \uparrow  \mathds{1}_{\Omega} N^{ij}(\d t )$.
Since the upper and lower bound are equal, we can conclude that
\begin{align*}
\mathds{1}_{\{\Lambda^{i \largecdot}(t-)< \infty\}} 	I^i(t-)\Lambda^{ij}(\d t) = \mathds{1}_{\{\Lambda^{i \largecdot}(t-)< \infty\}}  I^i(t-) \frac{\E[ N^{ij}(\d t )]}{\mathbb{E}[I^i(t-)] },\quad  t>0,
\end{align*}
almost surely, and because of the assumption \eqref{NoAmbioguityTransitionRates} we even have
\begin{align*}
\mathds{1}_{\{\Lambda^{i \largecdot}(t-)< \infty\}} \Lambda^{ij}(\d t) = \mathds{1}_{\{\Lambda^{i \largecdot}(t-)< \infty\}}   \frac{\E[ N^{ij}(\d t )]}{\mathbb{E}[I^i(t-)] },\quad  t>0,
\end{align*}
almost surely. Therefore, for the deterministic transition rates $\tilde{\Lambda}$ defined by  
\begin{align*}
\tilde{\Lambda}^{ij}(\d t) = \frac{\E[ N^{ij}(\d t )]}{\mathbb{E}[I^i(t-)] },\quad  i,j \in \mathcal{Z},\, i\neq j,
\end{align*}
and $\tilde{\Lambda}^{ij}(r):=0$ for $r=0$ and all time points that are poles of this function, we have that  $\tilde{\Lambda}-\tilde{\Lambda}(0)=\Lambda -\Lambda(0)$ $\mathbb{P}$-almost surely.  So there exist an $\Omega' \in \mathcal{F}_{\infty}$ with $\mathbb{P}[\Omega']=1$ such that $\tilde{\Lambda}-\tilde{\Lambda}(0)=\Lambda -\Lambda(0)$ everywhere on $\Omega'$. According to the explicit construction of $\mathbb{P}$ in the proof of Theorem \ref{TheoremCanoncialSpace}, the  restricted measures $\mathbb{P}|_{\Omega'}$ and $\tilde{\mathbb{P}}|_{\Omega'}$ generated by $(\alpha, \Lambda)$ and   $(\alpha, \tilde{\Lambda})$ are equal, see \eqref{DefOfP}. In particular, we have $\tilde{\mathbb{P}}[\Omega']= \mathbb{P}[\Omega]=1$ so that 
\begin{align*}
\tilde{\mathbb{P}}[A] &= \tilde{\mathbb{P}}[A \cap \Omega']= \mathbb{P}[A\cap \Omega']=\mathbb{P}[A], \quad A\in \mathcal{F}_{\infty}. \qedhere
\end{align*} 
\end{proof}
\begin{proposition}[Absolutely continuous modeling]\label{PropAbsolContinModel}
Let $i,j\in\mathcal{Z}$, $j \neq i$. The following conditions are equivalent:
\begin{itemize}
	\item[(a)] $\Lambda^{ij}$ is left-differentiable everywhere and continuously differentiable in between jump times of $Z${.}
	\item[(b)] The process $\mu^{ij}$ defined by 
	\begin{align*}
		\mu^{ij} (t) :=  \lim_{s \uparrow t } \frac{\mathbb{P}_{s}^i[Z(t)=j]}{t-s}, \quad t \in (0,\infty), 
	\end{align*}
	exists, is continuous in between jump times of $Z$, and satisfies the equation
	\begin{align*}
		\mu^{ij} (t) =   \lim_{u \downarrow t } \frac{\mathbb{P}_{t}^i[Z(u)=j]}{u-t}, \quad Z(t-)=Z(t),
	\end{align*}
	assuming that this limit exists.
\end{itemize}
Under both conditions, we moreover have 
\begin{align*}
	\Lambda^{ij}(\d t ) = \mu^{ij} (t)  \d t, \quad t \in [0,\infty).
\end{align*}
\end{proposition}
\begin{proof}
		The property (i) of Theorem \ref{TheoremCanoncialSpace} and the construction \eqref{pnipnij}  imply that
	\begin{align*}
		\lim_{t \downarrow s}\frac{\mathbb{P}_{s}^i[Z(t \wedge \tau(s))=j](\omega) }{t-s} &= \lim_{t \downarrow s}\frac{1}{t-s}\int_{(s,t]}  p_s^i(u-)(\omega) \Lambda^{ij}(\d u)(\omega_s^i) \\
		& =  	\lim_{t \downarrow s}\frac{\Lambda^{ij}(t)(\omega_s^i)-\Lambda^{ij}(s)(\omega_s^i)}{t-s}
	\end{align*}
	since the integrand converges to $1$ for $t \downarrow s$, uniformly for sufficiently small $t$. 
By similar arguments and  Proposition \ref{Proposition:CanoncialSpaceProperties}, we get
\begin{align*}
	\lim_{t \downarrow s}\frac{1}{t-s}\mathbb{P}_s^i[ \tau(\tau(s)) \leq t](\omega) & =\lim_{t \downarrow s}\frac{1}{t-s}\mathbb{E}_s^i[ \mathbb{P}_{\tau(s)}^{Z(\tau(s))}[ \tau(\tau(s)) \leq t](\omega)\\
	& = \lim_{t \downarrow s}\frac{1}{t-s}\sum_{j:j \neq i}\int_{(s,t]}  \mathbb{P}_u^j[\tau(u)\leq t ](\omega_s^i) p_s^i(u-)(\omega) \Lambda^{ij}(\d u)(\omega_s^i) \\
	& = \lim_{t \downarrow s}\frac{1}{t-s}\sum_{j:j \neq i}\int_{(s,t]} \big(1 - p_s^i(u-)(\omega_s^i)\big) p_s^i(u-)(\omega) \Lambda^{ij}(\d u)(\omega_s^i) \\
	& =0
\end{align*}
	since the integrand converges to $0$ for $t \downarrow s$, uniformly for sufficiently small $t$. Combining both results yields that 
	\begin{align}\label{Eq:muijRepresentationRight} \begin{split}
	\lim_{t \downarrow s}\frac{\mathbb{P}_{s}^i[Z(t )=j](\omega) }{t-s}	& =  	\lim_{t \downarrow s}\frac{\Lambda^{ij}(t)(\omega_s^i)-\Lambda^{ij}(s)(\omega_s^i)}{t-s}\\
	& = 	\lim_{t \downarrow s}\frac{\Lambda^{ij}(t)(\omega)-\Lambda^{ij}(s)(\omega)}{t-s} ,\quad Z(s-)=Z(s),
\end{split}\end{align}
where the last equation uses the right-continuity of $Z$ and  the Assumption \ref{as:tech_lambda}(a).
Since $Z$ is a càdlàg process, for each $\omega \in \Omega$ there exists an $\varepsilon_{\omega}>0$ such that  $\omega_s^i=\omega_t^i$ for  $s \in (t- \varepsilon_{\omega},t]$. By using the same asymptotic arguments as  above, we get
	\begin{align}\label{Eq:muijRepresentationLeft}\begin{split}
	\lim_{s \uparrow t}\frac{\mathbb{P}_{s}^i[Z(t )=j](\omega) }{t-s}	& =  	\lim_{t \uparrow s}\frac{\Lambda^{ij}(t)(\omega_t^i)-\Lambda^{ij}(s)(\omega_t^i)}{t-s}\\
	& =  	\lim_{t \uparrow s}\frac{\Lambda^{ij}(t)(\omega)-\Lambda^{ij}(s)(\omega)}{t-s},
\end{split}\end{align}
where the second equation uses the Assumption \ref{as:tech_lambda}(a).

Now, if statement (a) holds, then $\mu^{ij}$ exists and is continuous in between jump times because of \eqref{Eq:muijRepresentationLeft}. Moreover, in between jump times it equals  $ \lim_{h \downarrow 0 } \frac{\mathbb{P}_{t-h}^i[Z(t)=j]}{h}$  because of \eqref{Eq:muijRepresentationRight} and the differentiability of $\Lambda^{ij}$. In particular, we have  $\Lambda^{ij}(\d t ) = \mu^{ij} (t)  \d t$ since $\mu^{ij}$ is the derivative of $\Lambda^{ij}$ in between jump times and since there are at most countably many jump times.

If statement (b) holds, then $\Lambda^{ij}$ is left-differentiable  because of \eqref{Eq:muijRepresentationLeft}, and the left-derivative is continuous in between jump times. Moreover, in between jump times,  $\Lambda^{ij}$ has a right-derivative that equals the left-derivative because of  \eqref{Eq:muijRepresentationRight}.
\end{proof}
\begin{proposition}[Discrete modeling]\label{PropDiscreteModel}
	The  two following conditions are equivalent:
\begin{itemize}
	\item[(a)]  $\Lambda$ has pure jump paths with jumps only at integer times{.}
	\item[(b)] $\mathbb{E}^i_s[N^{ij}(\d t)] = 0$ for $n \leq s \leq t < n+1$, $i,j \in \mathcal{Z}$, $i \neq j$,  $n\in\mathbb{N}_0$.
\end{itemize}
Moreover, for $\Omega_d := \big\{(t_k,z_k)_{k \in \mathbb{N}_0} \in \Omega: t_k \in \mathbb{N}_0 \cup \{\infty\}, \, k \in \mathbb{N}_0\big\}$ the conditions imply that $ \mathbb{P}[\Omega_d]=1$ as well as
\begin{align*}
	\Delta \Lambda^{ij} (n+1)(\omega)=	\mathbb{P}_{n}^i[Z(n+1)=j](\omega), \quad  \omega \in \Omega_d,
\end{align*}
 for $ i,j \in \mathcal{Z}$, $i \neq j$, $n \in \mathbb{N}_0$.
\end{proposition}
\begin{proof}
For $n \leq s < n+1$, Theorem \ref{TheoremCanoncialSpace}(i)  implies that 
\begin{align*}
	\mathbb{P}_{s}^i[  \tau(s) \geq  n+1 ] & = 1- \sum_{j:j \neq i} \int_{(s,n+1) } \mathbb{E}_{s}^i[ \mathds{1}_{\{\tau(s)\geq u\}} N^{ij}( \d u ) ]\\
	& = 1- \sum_{j:j \neq i} \int_{(s,n+1)}  \mathbb{P}_{s}^i[  \tau(s) \geq u ] \, \Lambda^{ij}(\d u). 
\end{align*}
Note that if the interval $(s,n+1)$ contains reset points of $ \Lambda^{i \largecdot}$, then  $\mathbb{P}_{s}^i[  \tau(s) \geq u ]$ is nonzero only strictly before the smallest of these reset points.
If condition (a) holds, then from the latter equation we can conclude that $	\mathbb{P}_{s}^i[  \tau(s) \geq  n+1 ]=1 $, which implies (b). If condition (b) holds, then we can conclude that 
\begin{align*}
	\mathbb{E}_{s}^i[ \mathds{1}_{\{\tau(s)\geq u\}} N^{ij}( \d u ) ]&=0, \\
	\mathbb{P}_{s}^i[  \tau(s) \geq u ] \geq \mathbb{P}_{s}^i[  \tau(s) \geq n+1 ] &= 1, \quad s < u < n+1.
\end{align*}
This implies that  $\Lambda^{ij}(\d u)=0$ for $s < u < n+1$ and $u$ smaller than any reset point of   $ \Lambda^{i \largecdot}$ in $(s,n+1)$.
In particular, $\Lambda^{ij}$ is constant immediately before any reset point, which is impossible, so there are no reset points. That means that (a) holds.

Suppose that the equivalent conditions (a) and (b) hold. From the measurability assumption \eqref{NoAmbioguityTransitionRates} and Theorem \ref{TheoremCanoncialSpace}(i)  we get
\begin{align*}
 \Delta \Lambda^{ij}(n+1)(\omega)&= 	\Delta \Lambda^{ij}(n+1)(\omega_n^i)\\
 & =   \frac{\mathbb{E}_{n}^i [\mathds{1}_{\{\tau(n) \geq n+1\}}(N^{ij}(n+1)-N^{ij}(n)) ](\omega)}{\E_n^i[\mathds{1}_{\{\tau(n) \geq n+1\}}](\omega)}\\
 & = \mathbb{P}_{n}^i[Z(n+1)=j](\omega) ,\quad   \omega \in \Omega_d.
\end{align*}
 Moreover, Proposition \ref{Proposition:CanoncialSpaceProperties}(b) together with condition (b) implies that 
 $$ \mathbb{E}[N^{ij}(\d t)] = 0, \quad n \leq  t < n+1,\, i,j \in \mathcal{Z},\, i \neq j,\, n\in\mathbb{N}_0,$$
 from which we can conclude that   $\mathbb{P}[\tau_n \in (0,\infty) \setminus \mathbb {N} ]=0$, $n \in \mathbb{N}_0$. Thus, 
 $$ \mathbb{P}[\Omega \setminus \Omega_d]  \leq \sum_{n=0}^{\infty} \mathbb{P}[\tau_n \in (0,\infty) \setminus \mathbb {N} ]=0,$$
 which completes the proof.
\end{proof}

\section{Stochastic Kolmogorov backward equation}\label{sec:kolmogorov}

This section studies the time-dynamics of  $$ s \mapsto  \mathbb{P}_{s}^i [A ]$$ for any state $i \in \mathcal{Z}$ and events $A $ that can be observed in finite time. Recall that $\mathbb{P}_{s}^i [A ]$  can be interpreted as a version of the statewise conditional probability $\mathbb{P}[A | \mathcal{F}_{s-}, Z(s)=i]$, see Theorem \ref{TheoremCanoncialSpace}(i).

We say that a process $Y$ is bounded on finite intervals if  for each finite interval $I \subset [0,\infty)$ the mapping $Y : I \times \Omega \rightarrow \mathbb{R}$ is bounded. 
\begin{definition}
	Let  $\mathcal{Y}(\Lambda)$ denote the set of all $\mathcal{F}^-$-adapted, jointly measurable, multivariate  processes $(Y^i)_{i\in \mathcal{Z}}$ that are bounded on finite intervals, and such that the paths of $I^i(t-)Y^i(\d t)$ are càdlàg and of finite variation on each compact interval that contains no reset points of $\Lambda^{i \largecdot}$. 
\end{definition}
\begin{theorem}[stochastic Kolmogorov backward equation]\label{GeneralizKolmogBackwardEq} Let  $A \in \mathcal{F}_{T}$ for $T < \infty$, and let  $P^i$, $i \in \mathcal{Z}$, be given stochastic processes.  The two following statements are equivalent:
\begin{itemize}
	\item[(i)] The multivariate process $(P^i)_{i \in \mathcal{Z}}$  satisfies $P^i(t)= \mathbb{P}_{t}^i[A]$, $t \in [0,T]${.}
	\item[(ii)]  The multivariate process $(P^i)_{i \in \mathcal{Z}}$ is a solution in  $\mathcal{Y}(\Lambda)$ for the backward equation
\begin{align}\label{KolmBackward2}
	0 = \mathds{1}_{\{   \Lambda^{i\largecdot}(t-)< \infty\}} I^i(t-) \bigg( P^i(\d t)+\sum_{j:j\neq  i} \big(P^j(t)-P^i(t) \big) \Lambda^{ij}(\d t)\bigg) 
\end{align}
with terminal value $P^i(T)= \mathbb{P}_T^i[A]$. 
\end{itemize} 
If for a  subset $\mathcal{Z}_0 \subset \mathcal{Z}$ we have $\Lambda^{ij}=0$ for all $i \in \mathcal{Z}_0$ and $j \not\in \mathcal{Z}_0$, then the equivalence is also true for the sub-process $(P^i)_{i  \in \mathcal{Z}_0}$ only.
\end{theorem}
We interpret the right hand side of Equation~\eqref{KolmBackward2} as a $\sigma$-finite measure, and the equation tells us that this measure is zero on the set of reset points, and it is zero on any compact interval that does not contain  reset points. There exists a countable number of such compact intervals that covers the whole set $\{  t \in [0,\infty)  : \Delta \Lambda^{i\largecdot}(t)\geq 0\}$. So the measure is zero everywhere on $[0,\infty)$.
\begin{example}[Markov process]\label{ex:Markov_KolmogorovBackward}
By Proposition~\ref{prop:Markov_LambdaDeterministic}, a deterministic $\Lambda$ corresponds to a Markov model. In that case, the canonical probability kernels restricted to future events are also deterministic. Thus if $A \in \sigma(Z_T)$, Equation~\eqref{KolmBackward2} can be simplified to
\begin{align*}
	0 = \mathds{1}_{\{   \Lambda^{i\largecdot}(t-)< \infty\}}  \bigg( P^i(\d t)+\sum_{j:j\neq  i} \big(P^j(t)-P^i(t)\big) \Lambda^{ij}(\d t)\bigg).
\end{align*}
If there are also no reset points, {we conclude that
\begin{align*}
	 P^i(\d t) = -\sum_{j:j\neq  i} \big(P^j(t)-P^i(t)\big) \Lambda^{ij}(\d t),
\end{align*}
which recovers the classic Kolmogorov backward equation.}
\end{example}
\begin{proof}[{Proof of Theorem~\ref{GeneralizKolmogBackwardEq}}] At first we show that (i) implies (ii). The boundedness of $P^i$ follows directly from the fact that probabilities cannot be greater than one. 
	In the proof of  Theorem \ref{TheoremCanoncialSpace} we already showed that the probability kernels $p_s^{ij}(t)(\omega)$ are jointly measurable as mappings of $(s,\omega,t)$, see the arguments below \eqref{Ptaun+1Fs-Zsi}. For each $A \in \mathcal{F}_{\tau_n}$ and $n \in \mathbb{N}$, the mapping $ (s,i,\omega) \mapsto \mathbb{P}_s   ^i[A](\omega)$ is defined  by repeated Lebesgue-Stieltjes integration of the probability kernels with respect to the argument $t$ and summation over $i$, see  \eqref{DefOfcP}, so Fubini's theorem yields  joint measurability with respect to $(s,\omega)$. Since $\bigcup_{n \in \mathbb{N}_0} \mathcal{F}_{\tau_n}$ contains a generator of $\mathcal{F}_{T}$,  we also have the  joint measurability property for any $A \in  \mathcal{F}_{T}$.
	The constructions \eqref{pnipnij} and \eqref{DefOfcP}  imply that  $\mathbb{P}_s^i[A](\omega)=\mathbb{P}_s^i[A](\omega_s^i)$, $\omega \in \Omega$, which means that $\mathbb{P}_s^i[A]$ is $\mathcal{F}_{s-}$-measurable. Proposition \ref{Proposition:CanoncialSpaceProperties}(i) yields  that $P^i(T)= \mathds{1}_A$, $i \in \mathcal{Z}$. 
	 It remains to show that \eqref{KolmBackward2} holds. 	
Let $\omega=((t_l,z_l))_{l \in \mathbb{N}_0} \in \Omega$ be arbitrary but fixed.   The definition of  $	p_s^i$  implies that 
\begin{align}\label{IntEqForpnist}
	p_s^i(u-)(\omega_{s}^i) = 1- \sum_{k : k \neq i}\int_{(s,u)}  p_v^i(u-)(\omega_{s}^i)  \Lambda^{ik}(\d v)(\omega_{s}^i), \quad s < u < \rho_s^i (\omega).
\end{align}
From  \eqref{ProjectionProperty2} and the definitions \eqref{DefOfcP} and \eqref{pnipnij}, we obtain the equation
	\begin{align}\label{RepresPsiA}\begin{split}
		\mathbb{P}_{s}^i[ A ](\omega)&= \mathbb{E}_{s}^i\big[  	 \mathbb{P}_{\tau(s)}^{Z(\tau(s))}[A] \big](\omega)\\
		& =   \sum_{j:j \neq i} \int_{(s,\infty] }  \mathbb{P}_{u}^j[A](\omega_{s}^i) \, p_s^{ij}(\d u)( \omega_{s}^i)\\
		& =   \sum_{j:j \neq i}  \int_{(s,\infty] } \mathbb{P}_{u}^j[A](\omega_{s}^i) \, p_s^{i}( u-)(\omega_{s}^i) \, \Lambda^{ij}(\d u)(\omega_{s}^i).
\end{split}	\end{align}
	By replacing $ p_s^{i}(u-)(\omega_{s}^i)$ by \eqref{IntEqForpnist} and applying Tonelli's theorem, we obtain
	 \begin{align*}
	 	&\mathbb{P}_{s}^i[ A ](\omega)- \mathbb{P}_{t}^i[ A ](\omega)\\
	 	&=   \sum_{j:j \neq i}  \int_{(s,\infty] }\mathds{1}_{s< u \leq t}\,  \mathbb{P}_{u}^j[A](\omega_{s}^i) \, \Lambda^{ij}(\d u)(\omega_{s}^i)\\
	 	  &\quad -\sum_{j:j \neq i}  \int_{(s,\infty] }  \mathbb{P}_{u}^j[A](\omega_{s}^i) \sum_{k : k \neq i}\int_{(s,\infty]} (\mathds{1}_{s<v<u}-\mathds{1}_{t<v<u}) p_v^i(u-)(\omega_{s}^i)  \Lambda^{ik}(\d v)(\omega_{s}^i)\, \Lambda^{ij}(\d u)(\omega_{s}^i)\\
	 	&=  \sum_{j:j \neq i}  \int_{[0,\infty] }\mathds{1}_{s< u \leq t}\,  \mathbb{P}_{u}^j[A](\omega_{s}^i) \, \Lambda^{ij}(\d u)(\omega_{s}^i)\\
		&\quad - \sum_{k : k \neq i}\int_{[0,\infty]} \mathds{1}_{s<v \leq t}   \sum_{j:j \neq i}\int_{[0,\infty] }  \mathds{1}_{v<u}\,  \mathbb{P}_{u}^j[A](\omega_{s}^i) p_v^i(u-)(\omega_{s}^i)  \Lambda^{ij}(\d u)(\omega_{s}^i)  \Lambda^{ik}(\d v)(\omega_{s}^i)
	\end{align*}
		 for $ s < t \leq  \tau(s)(\omega)$ and $t < \rho_s^i(\omega)$.
	Since   $(\omega_{s}^i)_{v}^i= \omega_{s}^i$ for $v > s$, the inner integral in the latter line equals $	\mathbb{P}_{v}^i[ A ](\omega_{s}^i)$, see \eqref{RepresPsiA}, so that we can conclude that
		\begin{align*}
			&\mathbb{P}_{s}^i[ A ](\omega)- \mathbb{P}_{t}^i[ A ](\omega)	\\
			&=  \sum_{j:j \neq i} \int_{(s,t]}   \mathbb{P}_{u}^j[A](\omega_{s}^i) \, \Lambda^{ij}(\d u)(\omega_{s}^i)- \sum_{k : k \neq i}\int_{(s, t]} \mathbb{P}_{u}^i[A](\omega_{s}^i)   \Lambda^{ik}(\d u)(\omega_{s}^i)\\
			&=  \sum_{j:j \neq i} \int_{(s,t]}  \big(  \mathbb{P}_{u}^j[A](\omega) - \mathbb{P}_{u}^i[A](\omega)\big)\, \Lambda^{ij}(\d u)(\omega) 
				\end{align*}
				for $ s < t \leq  \tau(s)(\omega)$ and $t < \rho_s^i(\omega)$, using the fact that $ \mathbb{P}_{u}^i[A](\omega_{s}^i)= \mathbb{P}_{u}^i[A](\omega) $ and $\Lambda^{ik}(u)(\omega_{s}^i)=\Lambda^{ik}(u)(\omega)$ for $u \in (s,t]$ and $\omega \in  \{\tau_n \leq s < t \leq \tau_{n+1}\}$ due to their  $\mathcal{F}^-$-adaptedness and \eqref{GnFtEquivalence}. 
		All in all, this verifies that (a) implies (b).  If  $\Lambda^{ij}=0$ for all $i \in \mathcal{Z}_0$ and $j \not\in \mathcal{Z}_0$ for a subset $\mathcal{Z}_0 \subset \mathcal{Z}$, then the subset of equations \eqref{KolmBackward2} on $ \mathcal{Z}_0$ depends only on the sub-process $(P^i)_{i  \in \mathcal{Z}_0}$.
			
		Now we show that (i) implies (ii).  
		Integration by parts yields that 
		\begin{align*}
			&\sum_i I^i(s) P^i(s)  - \sum_i I^i(t)P^i(t)\\
			& = \sum_{i,j:j \neq i} \int_{(s,t]}  (P^j(u)-P^i(u)) N^{ij}(\d u) +  \sum_i \int_{(s,t]}  I^i (u- ) P^i(\d u )
		\end{align*}
		 for $ s < t \leq  \tau(s)(\omega)$ and $t < \rho_s^i(\omega)$.
		We subtract  the equation \eqref{KolmBackward2} and rearrange the terms in order to  arrive at the equation
		\begin{align*}
			&\sum_i I^i(s) P^i(s) - \sum_i I^i(t)P^i(t)\\
			&=  \sum_{j,j:j \neq i} \int_{(s,t]}   (P^j(u)-P^i(u)) \big( N^{ij}(\d u) - I^i(u-) \Lambda^{ij}(\d u )\big)
		\end{align*}
		for $ s < t \leq  \tau(s)(\omega)$ and $t < \rho_s^i(\omega)$. 
		By taking the expectation $\mathbb{E}_{s}^i[\cdot]$ on both sides,  applying Proposition \ref{PropMartingaleProperty},  and using the fact that $\mathbb{P}_s^i [\tau(s) < \rho_s^i]=1$, we obtain 
		\begin{align*}
			&\mathbb{E}_{s}^i\Big[ \sum_j I^j(s) P^j(s)  - \sum_j I^j(\tau(s))P^j(\tau(s))\Big]=0.
		\end{align*}
		By repeating this argument along the sequence of jump times $(\tau_n)_{n \in \mathbb{N}_0}$, which converges to $\lim_{n \rightarrow \infty }\tau_n = \infty> T$ by the definition of $\Omega$,  and using the tower property of conditional expectations according to Proposition \ref{Proposition:CanoncialSpaceProperties}(b), we even get 
			\begin{align}\label{ExpectZeroForPDynamics}
			&\mathbb{E}_{s}^i\Big[ \sum_j I^j(s) P^j(s)  - \sum_j I^j(T)P^j(T)\Big]=0, \quad s \leq T.
		\end{align}
		By applying Proposition \ref{Proposition:CanoncialSpaceProperties}(i) and using the equation  $\sum_j I^j(T)P^j(T)= \mathds{1}_A$, we can simplify the latter equation as
		\begin{align*}
			& P^i(s) - \mathbb{E}_{s}^i[ \mathds{1}_A ], \quad s \leq T,
		\end{align*}
		which verifies (i).   	If  $\Lambda^{ij}=0$ for all $i \in \mathcal{Z}_0$ and $j \not\in \mathcal{Z}_0$ for a subset $\mathcal{Z}_0 \subset \mathcal{Z}$, then
		 we necessarily have  $\mathbb{P}_s^{i}[Z(t)=j]=0$,    $i \in \mathcal{Z}_0$, $j \in \mathcal{Z}\setminus\mathcal{Z}_0$, $0 \leq s \leq t$,  according to the definition of $\mathbb{P}_s^{i}$ in  \eqref{pnipnij} and \eqref{DefOfcP}.  Therefore, in case of $i \in \mathcal{Z}_0$, in equation \eqref{ExpectZeroForPDynamics} we can restrict the sums over $j$ to the subset $\mathcal{Z}_0$. This means that the $\mathcal{Z}_0$-subset of equations \eqref{KolmBackward2} already suffices to obtain $P^i(s) - \mathbb{E}_{s}^i[ \mathds{1}_A ]$, $i \in \mathcal{Z}_0$. 
\end{proof}
\begin{example}
	In the setting of Proposition~\ref{PropAbsolContinModel}, Equation~\eqref{KolmBackward2} almost surely corresponds to the  differential equation
	\begin{align*}
			0 =  I^i(t-) \bigg( \frac{\d}{\d t}P^i(t)+\sum_{j:j\neq  i} \big(P^j(t)-P^i(t) \big) \mu^{ij}( t)\bigg) .
	\end{align*}
\end{example}
\begin{example}
	In the setting of Proposition~\ref{PropDiscreteModel}, Equation~\eqref{KolmBackward2} almost surely corresponds to the backward recursion equation
	\begin{align*}
		{}_{T-(n-1)}p^{ik}_{n-1} =   {}_{T-n}p^{ik}_{n}  + \sum_{j: j \neq i}   \big( {}_{T-n}p^{jk}_{n} - {}_{T-n}p^{ik}_{n}\big) q^{ij}_{n-1}  , \quad Z(n-1)=i,
	\end{align*}
	for $n \in \{0,1, \ldots, T-1\}$  with terminal condition $ {}_{0}p^{ik}_{T}= \mathds{1}_{i=k}$, where 
	\begin{align*}
			{}_{T-n}p^{ik}_n:= \mathbb{P}_{n}^i[ Z(T)=k].
	\end{align*} 
\end{example}

\section{State-wise conditional expectation processes}\label{sec:statewise}

This section studies the time-dynamics of  $$s \mapsto \mathbb{E}_{s}^i [ Y(T)- Y(s) ]$$ for any state $i \in \mathcal{Z}$ and an $\mathcal{F}$-adapted càdlàg process $Y$ that is sufficiently integrable. We first introduce  a large class of integrable processes $Y$, and then we study the path properties of the above expectation process.
\begin{definition}
	Let $\mathbb{Y}^+$ denote the set of all $\mathcal{F}$-adapted, nonnegative, nondecreasing, univariate  càdlàg processes
	$Y$ for which there exists a growth bound  
	\begin{align}\label{GrowthBoundCond}
		Y(t) \leq g(t) \Big( 1+ \sum_{i,j:i \neq j} N^{ij}(t )\Big)^n, \quad t \geq 0,
	\end{align}
	for some function  $g:[0, \infty) \rightarrow [0,\infty)$ and a positive integer $n \in \mathbb{N}$; both $g$ and $n$ may depend on $Y$. By $\mathbb{Y}$ we denote the set of all processes that can be generated as a difference of processes from $\mathbb{Y}^+$.	
\end{definition}
Note that $Y \in \mathbb{Y}$ has paths of finite variation on finite intervals since it equals the difference of monotone processes.
\begin{proposition} \label{BoundedExpectedN}
	For  $i,j \in \mathcal{Z}$, $i \neq j$,  $ T \in  [0,\infty)$, and $n \in \mathbb{N}$, it holds that 
	\begin{align*}
		\sup_{0 \leq s \leq t \leq T}\sup_{\omega \in \Omega } \mathbb{E}_{s}^i\big[  \big( N^{ij}(t) - N^{ij}(s)\big)^n\big](\omega) < \infty.
	\end{align*}
	 In particular, for $Y \in \mathbb{Y}$ we have 
	 \begin{align*}
	 	\sup_{0 \leq s \leq t \leq T}\sup_{\omega \in \Omega } \mathbb{E}_{s}^i\big[  |Y(t)-Y(s)|\big](\omega) < \infty.
	 \end{align*}
\end{proposition}
\begin{proof}
	Let $J \subset \{(j,k) \in \mathcal{Z}^2: j  \neq k\}$ be the subset of transitions for which the  corresponding  transition rates  are bounded on finite intervals, i.e.
	\begin{align*}
		\sup_{0 \leq s \leq t \leq T} \big|\Lambda^{jk}(t)-\Lambda^{jk}(s) \big| \leq \big|\Lambda^{jk}(T)-\Lambda^{jk}(0) \big| \leq C < \infty, \quad (j,k) \in J,
	\end{align*}
	for a deterministic constant $C$.
	Thus, by equation \eqref{UpperEstimateForJumps} and \eqref{CompensatorEquation}  we  get the finite upper bound
	\begin{align*}\
		\mathbb{E}_{s}^i\bigg[ \sum_{j,k: j\neq k}\big( N_{jk}(t)-N_{jk}( s)\big)\bigg] &\leq | \mathcal{Z}|-1+ | \mathcal{Z}|\,	\mathbb{E}_{s}^i\bigg[ \sum_{(j,k) \in J}\big(\Lambda^{jk}(t)-\Lambda^{jk}(s)\big)\Big] \\
		& \leq  | \mathcal{Z}|-1+ | \mathcal{Z}| \, | \mathcal{Z}| (   | \mathcal{Z}| -1) C
	\end{align*}
	uniformly	for all $\omega \in \Omega$ and all $0 \leq s \leq t\leq T$.	 Since the addends on the left hand side are all non-negative, the first assertion follows for $n=1$. We generalize the result to any $n \in \mathbb{N}$ by induction. Suppose that the statement of the proposition is true for all $n \leq m \in \mathbb{N}$. Since $(a+1)^{m+1} - a^{m+1} = \sum_{l=0}^m \binom{m+1}{l} a ^l $ for any $a \in [0,\infty) $, we have 
	$$  \big( N^{jk}(t)-N^{jk}( s)\big)^{m+1} =  \sum_{l=0}^m \binom{m+1}{l} \int_{(s,t]} \big( N^{jk}(u-)-N^{jk}( s)\big)^{l} N^{jk} (\d u) .  $$
	From Proposition \ref{PropMartingaleProperty},  the monotone convergence theorem, and the monotony of $\E_s^i[\cdot]$, we can conclude that
	\begin{align*} 	
		 \mathbb{E}_{s}^i\Big[ \big( N^{jk}(t)-N^{jk}( s)\big)^{m+1} \Big] &=   \sum_{l=0}^m \binom{m+1}{l}  \mathbb{E}_{s}^i\bigg[  \int_{(s,t]} \big( N^{jk}(u-)-N^{jk}( s)\big)^{l} I^j(u-) \Lambda^{jk} (\d u) \bigg]  \\
		 &  \leq  C \sum_{l=0}^m \binom{m+1}{l}  \mathbb{E}_{s}^i\Big[  \big( N^{jk}(t-)-N^{jk}( s)\big)^{l}  \Big], \quad (j,k) \in J. 
	\end{align*}
	Because of the induction assumption, the latter term has an upper bound, uniformly in the parameters $\omega \in \Omega$ and $0\leq s \leq t \leq T$. By applying \eqref{UpperEstimateForJumps} and using the fact that $(a_1 +\cdots + a_r)^{m+1} \leq  r^m (a_1^{m+1} + \cdots + a_r^{m+1})$ for any $a_1, \ldots, a_r \in [0,\infty)$ and $r \in \mathbb{N}$,  we also get  a uniform finite upper bound for  $\sum_{j,k:j \neq k} \mathbb{E}_{s}^i[ ( N^{jk}(t)-N^{jk}( s))^{m+1}] $. The non-negativity of the addends  implies that the  upper bound applies also for the individual addends $\mathbb{E}_{s}^i[ ( N^{jk}(t)-N^{jk}( s))^{m+1}]$, $j,k \in \mathcal{Z}$, $j \neq k$. This completes the induction, so the first assertion is verified for all $n \in \mathbb{N}$. 
	
	The second assertion  follows directly from the first assertion by applying assumption \eqref{GrowthBoundCond}.
\end{proof}
\begin{proposition}\label{ContainerTheoremE}
	For $Y \in \mathbb{Y}$, the processes  $Y^i$, $i \in \mathcal{Z}$, defined by 
	\begin{align}\label{DefOfYi}
		Y^i(t)(\omega):=Y(t)(\omega_t^i), \quad t \geq 0,\, \omega \in \Omega,
	\end{align}
	are 
	$\mathcal{F}^-$-adapted and jointly measurable, and the paths of $I^i(t-) Y^i(\d t)$ are   càdlàg  and of  finite variation on finite intervals. 
		It holds that
		\begin{align}\label{CanonicRepres}
			 Y(\d t) = \sum_{i} I^i(t-)  Y^i(\d t) + \sum_{i,j:i \neq j} (Y^j(t)-Y^i(t)) N^{ij}(\d t). 
		\end{align}
\end{proposition}
\begin{proof}
	By the definition \eqref{omegastopped}, the mapping $\omega  \rightarrow \omega_t^i$ is $\mathcal{F}_{t-}$-measurable, so the process $Y^i$ is $\mathcal{F}^-$-adapted.  
	The mapping $(t,\omega) \mapsto Y(t)(\omega)$ is measurable since the process $Y$ has  càdlàg paths. 
		 The mapping $(t,\omega) \mapsto  (t,\omega_t^i)$ is measurable as a mapping from the measurable space $([0,\infty) \times \Omega, \mathcal{B}([0,\infty))\otimes \mathcal{F}_{\infty})$ to the same  space, since it is  composed of simple functions and countably many case differentiations, see \eqref{omegastopped}. Since $Y^i$ is  a composition of the latter two mappings, it is also measurable as a mapping of $(t,\omega)$. 		 Since $I^i(t-) Y^i(\d t) = I^i(t-) Y(\d t)$, the paths of   $I^i(t-) Y^i(\d t) = I^i(t-) Y(\d t)$ are càdlàg and of finite variation on finite intervals.  
Since  $Y $ is $\mathcal{F}$-adapted, 
we have that
\begin{align}\label{BEqualsEB}
	Y(t) = \sum_i I^i(t) Y^i(t) ,\quad t \geq 0.
\end{align}
By applying integration by parts, we obtain \eqref{CanonicRepres}.
\end{proof}

\begin{theorem}\label{DefStatewiseCondExpect0}
	For  $Y \in \mathbb{Y}$, let  $Y^i$, $i \in \mathcal{Z}$, be defined by \eqref{DefOfYi}. 
		For given stochastic processes $E^i$, $i \in \mathcal{Z}$, the two following statements are equivalent:
		\begin{itemize}
			\item[(i)]  The multivariate process  $(E^i)_{i \in \mathcal{Z}}$ satisfies $E^i(t)= \mathbb{E}_{t}^i[Y(T)-Y(t)]$ for $t \in [0,T]${.}
			\item[(ii)]  The multivariate process  $(E^i)_{i \in \mathcal{Z}}$ is a solution of  $\mathcal{Y}(\Lambda)$ for the backward equation
			\begin{align}\label{ThieleEq0} \begin{split}
					0=\mathds{1}_{\{  \Lambda^{i\largecdot}(t-)< \infty\}} I^i(t-)  \bigg( E^i(\d t)   +  Y^i(\d t) + \sum_{j : j\neq i}(Y^{j}(t)-Y^i(t)+ E^j(t)-E^i(t)) \Lambda^{ij}(\d t)\bigg)
			\end{split}\end{align}
		 with terminal value $E^i(T)=0$. 
		\end{itemize}
		If for a  subset $\mathcal{Z}_0 \subset \mathcal{Z}$ we have $\Lambda^{ij}=0$ for all $i \in \mathcal{Z}_0$ and $j \not\in \mathcal{Z}_0$, then the equivalence is also true for the sub-process $(E^i)_{i  \in \mathcal{Z}_0}$ only.
\end{theorem}
We interpret the right hand side of Equation~\eqref{ThieleEq0} as a $\sigma$-finite measure, and the equation tells us that this measure is zero on the set of reset points, and it is zero on any compact interval that does not contain  reset points. There exists a countable number of such compact intervals that covers the whole set $\{  t \in [0,\infty)  : \Delta \Lambda^{i\largecdot}(t)\geq 0\}$. So the measure is zero everywhere on $[0,\infty)$.
\begin{proof}
	 At first, we show that (i) implies (ii). 
	 	Proposition \ref{BoundedExpectedN} yields that $E^{i}$ is bounded on finite intervals. 	For each $u \in [0,\infty)$, we define a stochastic process $E^i_u$  by 
	 $$E_u^{i}(t):=  \mathbb{E}_{t}^i[Y(T)-Y(u)], \quad  t \in [0 ,\infty).$$
	 Since the random variable $Y(T)-Y(u)$ can be represented as a limit of linear combinations of indicator random variables $\mathds{1}_A$, $A \in \mathcal{F}_{\infty}$, the joint measurability of $ \mathbb{E}_{t}^i[\mathds{1}_A]$ according to Theorem \ref{GeneralizKolmogBackwardEq} yields also the joint measurability of each process $E^i_u$, $u \geq 0$. 
	 The right-continuity of $Y$ and the dominated convergence theorem imply that 
	 $$ \lim_{u \downarrow t} E_u^i(t) = E^i(t).$$ 
	 So, for any sequence  $\mathcal{T}_n $, $n \in \mathbb{N}$, of partitions of $[0,t]$ with vanishing maximum step length for $n$ to infinity, we have
	 \begin{align*}
	 	E^i =  \lim_{n\rightarrow 0} \sum_{\mathcal{T}_n} \mathds{1}_{[t_l,t_{l+1})} E_{t_{l+1}}^i.
	 \end{align*}
	 Because of the latter limit representation, the joint measurability of  the processes $E_u^i$, $u \geq 0$, is inherited by  $E^i$. 	Since  $E_u^{i}$ is $\mathcal{F}^-$-adapted for each $u \in [0,\infty)$, see Theorem  \ref{GeneralizKolmogBackwardEq} and definition \eqref{DefExpectIntTheProof}, 
	 we have that 
	 $E^i(t)= E_t^{i}(t)$ is $\mathcal{F}_{t-}$-measurable for each $t \in [0,\infty)$, which means that $E^i$ is $\mathcal{F}^-$-adapted.    		
	 By using the fact that any random variable can be asymptotically approximated from below by simple random variables, the dominated convergence theorem and Theorem  \ref{GeneralizKolmogBackwardEq}  yield that
	 \begin{align}\label{Generat}\begin{split}
	 		E^i(t)-E^i(s)&= \E_{t}^i[Y(T)-Y(s)] - \E_{s}^i[Y(T)- Y(s)] +\E_{t}^i[Y(s)- Y(t)]\\
	 		&= \sum_{j:j \neq i}\int_{(s,t]}  \big(\E_{u}^j[Y(T)-Y(s)]-\E_{u}^i[Y(T)- Y(s)]\big)\Lambda^{ij}(\d u) + \E_{t}^i[Y(s)- Y(t)],\\
	 		&  \hspace{6cm}  t \in (s,\tau(s)] \cap (0, \rho^i_s), Z(s)=i.
	 \end{split}\end{align}
	 If we read the latter integral as a stochastic process in $t$, then this process has  càdlàg paths of finite variation on each compact interval that contains no reset points of $\Lambda^{i\largecdot}$. On the same intervals, the process $t \mapsto \E_{t}^i[Y(s)- Y(t)]$ has càdlàg paths due to the  dominated convergence theorem, and finite variation due to the  triangle inequality for the conditional expectations and by employing the  representation of $Y$ as a difference of two monotone processes from $\mathbb{Y}^+$.  This proves that  $(E^i)_{i \in \mathcal{Z}} \in \mathcal{Y}(\Lambda)$.
	
	 By applying Proposition \ref{Proposition:CanoncialSpaceProperties}(i) and \eqref{CanonicRepres} and using adaptedness properties, from equation \eqref{Generat} we can conclude that 
	\begin{align*}
		&\mathbb{E}_{t}^i[Y(T)-Y(t)] - \mathbb{E}_{s}^i[Y(T)-Y(s)]  - \int_{(s,t]} I^i(u-) Y^i(\d u)  \\
		&= \sum_{j : j\neq i}\int_{(s,t]} \big( \mathbb{E}_{u}^j[Y(T)-Y(s)]  -	\mathbb{E}_{u}^i[Y(T)-Y(s)]\big) \Lambda^{ij}(\d u),\\
		&= \sum_{j : j\neq i}\int_{(s,t]} \big(Y^j(u) -Y^i(u) + \mathbb{E}_{u}^j[Y(T)-Y(u)]  -	\mathbb{E}_{u}^i[Y(T)-Y(u)]\big) \Lambda^{ij}(\d u),\\
		& \hspace*{10cm} t \in (s, \tau(s)] \cap (s, \rho^i_s), Z(s)=i.
	\end{align*}
	This implies equation \eqref{ThieleEq}, but only  on  compact intervals that contain not reset points of $\Lambda^{i \largecdot}$.  There exists a countable number of these compact intervals that cover the whole set $\{  t \in [0,\infty)  : \Delta \Lambda^{i\largecdot}(t)\geq 0\}$. This verifies \eqref{ThieleEq}.  If  $\Lambda^{ij}=0$ for all $i \in \mathcal{Z}_0$ and $j \not\in \mathcal{Z}_0$ for a subset $\mathcal{Z}_0 \subset \mathcal{Z}$, then the subset of equations \eqref{ThieleEq0}  on $ \mathcal{Z}_0$ depends only on the sub-process $(E^i)_{i  \in \mathcal{Z}_0}$.

	Now we show that (ii) implies (i). Integration by parts yields that 
	\begin{align*}
		&\sum_i I^i(s)E^i(s)  - \sum_i I^i(t)E^i(t) \\
		& = \sum_{i,j:j \neq i} \int_{(s,t]} (E^j(u)-E^i(u)) N^{ij}(\d u) +  \sum_i \int_{(s,t]}  I^i (u- ) E^i(\d u ) 
	\end{align*}
	for  $t \in (s, \tau(s)] \cap (s, \rho^i_s)$, $i \in \mathcal{Z}$.
	We subtract the equation \eqref{ThieleEq} for all $i \in \mathcal{Z}$ and rearrange the terms in order to  arrive at the equation
	\begin{align*}
		&\sum_i I^i(s) E^i(s) - \sum_i I^i(t)E^i(t) \\
		&=  -\sum_i  \int_{(s,t]}  I^i(u-)  B(\d u) + \sum_{j,j:j \neq i} \int_{(s,t]}  (Y^j(u) -Y^i(u)+ E^j(u)-E^i(u)) \big( N^{ij}(\d u) - I^i(u-) \Lambda^{ij}(\d u )\big)
	\end{align*}
	for  $t \in (s, \tau(s)] \cap (s, \rho^i_s)$,  $i \in \mathcal{Z}$. By setting $t =\tau_{n+1}$, taking the expectation $\mathbb{E}_{s}^i[\cdot]$ on both sides, applying Proposition \ref{PropMartingaleProperty}, and using the fact that 
	\begin{align*}
		\mathbb{P}_{s}^i[\tau_{n+1} \geq \rho_{s}^i]&= p_{s}^i(\rho_{s}^i-)=0 \quad  \tau_n \leq s < \tau_{n+1},\zeta_{n}=i,
	\end{align*} 
	see definition  \eqref{DefOfcP}, we obtain
	\begin{align*}
		\mathbb{E}_{s}^i\bigg[\sum_j I^j(s)E^j(s)- \sum_j I^j(\tau_{n+1})E^j(\tau_{n+1}) \bigg]
		&= - \mathbb{E}_{s}^i[  Y(\tau_{n+1}) -Y(s)], \quad \tau_n \leq s < \tau_{n+1},\zeta_{n}=i.
	\end{align*}
	By applying the tower property of conditional expectations according to  Proposition \ref{Proposition:CanoncialSpaceProperties}(b), we moreover  get
	\begin{align*}
		\mathbb{E}_{s}^i\bigg[\sum_j I^j(\tau_n )E^j(\tau_n)-\sum_j I^j(\tau_{n+1})E^j(\tau_{n+1})\bigg] = -  \mathbb{E}_s^i[   Y(\tau_{n+1}) -Y(\tau_n)], \quad s < \tau_n .
	\end{align*}
	All in all, by  using Proposition \ref{Proposition:CanoncialSpaceProperties}(a) and  the assumption $E^j(T)=0$, $j \in \mathcal{Z}$,   the  two latter equations, and the fact that   $\lim_{n \rightarrow \infty }\tau_n = \infty> T$ according to the definition of $\Omega$, we obtain
	\begin{align*}
		E^i(s)  &= 	 \mathbb{E}_{s}^i\bigg[\sum_j I^j(s)E^j(s)- \sum_j I^j(T)E^j(T) \bigg]\\
		&= \mathbb{E}_{s}^i\bigg[ \sum_{l}   \int_{(s ,T]\cap (\tau_l, \tau_{l+1}]}   Y(\d u)\bigg] \\
		& =  \mathbb{E}_{s}^i[Y(T) -Y(s) ]
	\end{align*}
	for all $s \leq T$. 	If  $\Lambda^{ij}=0$ for all $i \in \mathcal{Z}_0$ and $j \not\in \mathcal{Z}_0$ for a subset $\mathcal{Z}_0 \subset \mathcal{Z}$, then
	we necessarily have  $\mathbb{P}_s^{i}[Z(t)=j]=0$,    $i \in \mathcal{Z}_0$, $j \in \mathcal{Z}\setminus\mathcal{Z}_0$, $0 \leq s \leq t$,  according to the definition of $\mathbb{P}_s^{i}$ in  \eqref{pnipnij} and \eqref{DefOfcP}.  Therefore, in case of $i \in \mathcal{Z}_0$, in the latter equation we can restrict the sums over $j$ to the subset $\mathcal{Z}_0$. This means that the $\mathcal{Z}_0$-subset of equations \eqref{KolmBackward2} already suffices to obtain $	E^i(s) = \mathbb{E}_{s}^i[Y(T) -Y(s) ] $, $i \in \mathcal{Z}_0$. 
\end{proof}

\begin{remark}[Martingale representation]
	We consider the martingale
	\begin{align*}
	X(t) =   \E[Y(T)| \mathcal{F}_t], \quad t \in [0,T].
	\end{align*} 
	As we do not use the usual hypotheses for our filtered probability space, we cannot use the standard arguments that ensure that any martingale has a càdlàg modification, but based on our canonical probability kernels we can define a càdlàg modification by
$ X := \sum_i I^i X^i$ for
\begin{align*}
X^i(t):= \mathbb{E}_t^i[Y(T)], \quad t \in [0,T].
\end{align*}
Suppose there are no reset points.  Integration by parts, the relation $X^i=E^i+Y^i$, $i \in \mathcal{Z}$, and~\eqref{ThieleEq0} yield the martingale representation 
	\begin{align*}
		X(t) =  Y(T) - \sum_{i,j : i \neq j} \int_{(t,T]} (X^j(u)-X^i(u)) M^{ij}(\mathrm{d}t), \quad t \in [0,T],
	\end{align*}
	where  $M = (M^{ij})_{i \neq j}$ is the collection of martingales given by
	\begin{align*}
		M^{ij}(\mathrm{d}t) := N^{ij}(\mathrm{d}t) - I^i(t-)\Lambda^{ij}(\mathrm{d}t), \quad t\geq0, 
	\end{align*}
	confer with Proposition~\ref{PropMartingaleProperty}. Unlike what we usually find in the literature, our martingale representation holds surely. In~\cite{Jacobsen2006} a sure martingale representation is also presented, but a càdlàg modification of $X$ is presumed already given, whereas we offer a construction of such a modification.
\end{remark}

\section{Stochastic Thiele equation}\label{sec:thiele}

An insurance cash flow $C$ is a  bidirectional cash flow that gives the difference of the cumulative benefits cash flow and the cumulative premiums cash flow. We generally assume that the cumulative  benefits cash flow and the cumulative premiums cash flow are processes from $\mathbb{Y}^+$, so that $C$ is an element of $\mathbb{Y}$. According to Proposition~\ref{ContainerTheoremE}, the insurance cash flow can be represented as
\begin{align*}
	C(t) = \sum_i \int_{[0,t]} I^i(u-) B^i(\d u) + \sum_{i,j:i \neq j} b^{ij}(u) N^{ij}(\d u), \quad t \geq 0,
	\end{align*}
for $B^i(u)(\omega):=C(u)(\omega_u^i)$ and $b^{ij}(u)(\omega):=C(u)(\omega_u^j)-C(u)(\omega_u^i)$. We call
\begin{align*}
(B,b)= ((B^i)_{i\in \mathcal{Z}}, (b^{ij})_{i,j \in \mathcal{Z}: i\neq j}))
\end{align*}
the  canonical representation of $C$ and interpret $B^i(u)$ as the aggregated sojourn payments on $[0,u]$ in state $i$ and $b^{ij}(u)$ as the transition payments for a transition from $i$ to $j$ at time $u$. 

{The accumulation of wealth is described through a \textbf{cumulative interest rate} $R$, which we assume is an $\mathcal{F}^-$-adapted element of $\mathbb{Y}$ such that $\Delta R > -1$ and such that $\inf_{\omega \in \Omega} R(\omega)$ is bounded on compacts. Based hereon, we define an $\mathcal{F}^-$-adapted càdlàg and strictly positive process $\kappa$ with paths of finite variation on compact intervals and the property that not only $\kappa$ itself but also its reciprocal $1/\kappa$ is bounded on finite intervals via
\begin{align*}
\kappa(t) = e^{R_c(t) - R_c(0)} \prod_{0 < s \leq t} \big(1 + \Delta R(s)\big), \quad t \geq 0,
\end{align*}
where $R_c$ denotes the continuous part of $R$. The interpretation of $\kappa$ is that of a savings account. Note that since the cumulative interest rate $R$ is an $\mathcal{F}^-$-adapted element of $\mathbb{Y}$, specializing Proposition~\ref{ContainerTheoremE} it can be represented as}
\begin{align*}
	R(t) = \sum_i \int_{[0,t]} I^i(u-) \Phi^i(\d u), \quad t \geq 0,
	\end{align*}
for $\Phi^i(u)(\omega):=R(u)(\omega_u^i)$. We call
\begin{align*}
\Phi= ((\Phi^i)_{i\in \mathcal{Z}})
\end{align*}
the  canonical representation of $R$ and interpret $\Phi^i(u)$ as the total sojourn accumulation of wealth on $[0,u]$ in state $i$.

The future discounted liabilities of the insurer at time $t$ and up until some finite horizon $T<\infty$ are given by 
	\begin{align}\label{DiscountedFutureLiabilities}
		L(t,T) := \int_{(t, T]} \frac{\kappa(t)}{\kappa(u)} B(\d u) , \quad 0 \leq t \leq T.
	\end{align}
The-state-wise prospective reserves are introduced in~\cite{Norberg1992}, with 
\begin{align}\label{PRE:DefStateWiseProspRes0}
	\E[ L(t,T) \, | \,Z(t)=i,   \mathcal{F}_{t-}], \quad  t \in [0,T], \, i \in \mathcal{Z},
\end{align}
as the state-wise prospective reserve in state $i$ at time $t$. This definition, however, has its flaws: it is pointwise almost surely only and its path properties are unclear, confer also with \cite{ChristiansenFurrer2021}. This was already noted in~\cite{Norberg1996}, but not rigorously resolved. The problem can be overcome by choosing for~\eqref{PRE:DefStateWiseProspRes0} the unique version
\begin{align}\label{PRE:DefStateWiseProspRes1}
	\E_t^i[ L(t,T) ],\quad  t \in [0,T], \, i \in \mathcal{Z},
\end{align}
which makes the definition of state-wise prospective reserves surely unique and comes with nice path properties.

In the following, we generally suppress the dependence of $L$ and derived quantities on $T$; that is, we write $L(t)$ in place of $L(t,T)$. This is solely to lessen the notational burden.

\begin{definition}\label{def:canonical_insurance_model}
	The tuple  $(\alpha, \Lambda, \Phi, B, b)$, consisting of the generator $(\alpha, \Lambda)$ for  the probability space $(\Omega, \mathcal{F}_{\infty}, \mathbb{P})$, the canonical interest rate representation $\Phi$, and the canonical insurance cash flow representation $(B,b)$ is called a \emph{canonical insurance model}.
\end{definition}
\begin{theorem}[Stochastic Thiele equation]\label{stochThiele} Let a canonical insurance model $(\alpha, \Lambda, \Phi, B, b)$ be given. 
For given stochastic processes $V^i$, $i \in \mathcal{Z}$, the two following statements are equivalent:
	\begin{itemize}
		\item[(a)]  The process  $(V^i)_{i \in \mathcal{Z}}$  satisfies  $V^i(t)= \mathbb{E}_{t}^i[L(t)]$, $t \in [0,T]${.}
		\item[(b)] The process  $(V^i)_{i \in \mathcal{Z}}$ is a solution in  $\mathcal{Y}(\Lambda)$ for the backward equation
		\begin{align}\label{ThieleEq} \begin{split}
				0=\mathds{1}_{\{  \Lambda^{i\largecdot}(t-)< \infty\}} I^i(t-)  \bigg(& V^i(\d t)   +  B^i(\d t) - V^i(t-) \Phi^i (\d t) \\&+ \sum_{j : j\neq i}(b^{ij}(t) + V^j(t)-V^i(t)) \Lambda^{ij}(\d t)\bigg)
		\end{split}\end{align}
		with  terminal value $V^i(T)=0$. 
	\end{itemize}
	If for a  subset $\mathcal{Z}_0 \subset \mathcal{Z}$ we have $\Lambda^{ij}=0$ for all $i \in \mathcal{Z}_0$ and $j \not\in \mathcal{Z}_0$, then the equivalence is also true for the sub-process $(V^i)_{i  \in \mathcal{Z}_0}$ only.	
\end{theorem}
\begin{example}[Markov model]\label{ex:Markov_Thiele}
In continuation of Example~\ref{ex:Markov_KolmogorovBackward}, suppose that $\Lambda$, $\Phi$, $B$, and $b$ are deterministic. In that case, the state-wise prospective reserves are also deterministic, and Equation~\eqref{ThieleEq} reads
\begin{align*}
0=\mathds{1}_{\{  \Lambda^{i\largecdot}(t-)< \infty\}} \bigg( V^i(\d t)   +  B^i(\d t) - V^i(t-) \Phi^i (\d t)+ \sum_{j : j\neq i}(b^{ij}(t) + V^j(t)-V^i(t)) \Lambda^{ij}(\d t)\bigg).
\end{align*}
If there are no reset points, we recover the classic Thiele equation:
\begin{align*}
 V^i(\d t) &=  V^i(t-) \Phi^i (\d t) -  B^i(\d t) - \sum_{j : j\neq i}(b^{ij}(t) + V^j(t)-V^i(t)) \Lambda^{ij}(\d t). \qedhere
\end{align*} 
\end{example}
\begin{proof} 
	Theorem \ref{DefStatewiseCondExpect0} yields for $W^i(t):= V^i(t)/\kappa(t) $ the equivalence of $W^i(t)=\E_t^i[L(t)/\kappa(t)]$ and the  backward equation
	\begin{align*} 
			0=\mathds{1}_{\{  \Lambda^{i\largecdot}(t-)< \infty\}} I^i(t-)  \bigg( W^i(\d t)   + \frac{1}{\kappa(t)} B^i(\d t) + \sum_{j : j\neq i}\Big(\frac{1}{\kappa(t)}b^{ij}(t)+ W^j(t)-W^i(t)\Big) \Lambda^{ij}(\d t)\bigg).
	\end{align*}
	Integration by parts yields
	\begin{align*}
		   I^i(t-)  V^i(\d t)
		   &=
		   I^i(t-) \kappa(t) W^i(\d t)  + I^i(t-) W^i(t-)  \kappa(\d t) \\
		   &=
		   I^i(t-) \kappa(t) W^i(\d t)  + I^i(t-) V^i(t-)  \Phi^i(\d t),
	\end{align*}
	and by multiplying this equation with $1/\kappa(t)$ based on the Radon-Nikodym theorem, we obtain
	\begin{align*}
		 I^i(t-) W^i(\d t)   =   I^i(t-) \frac{1}{\kappa(t)} V^i(\d t) - I^i(t-) \frac{1}{\kappa(t)} V^i(t-) \Phi^i (\d t).
	\end{align*}
	Therefore, the backward equation for $W^i$ is equivalent to~\eqref{ThieleEq}. This proves the equivalence of (a) and (b), since the transformation  $W^i(t):= V^i(t)/\kappa(t) $, $i \in \mathcal{Z}$, stays in $\mathcal{Y}(\Lambda)$ due to the boundedness assumptions for $\kappa$ and $1/\kappa$ on $[0,T]$. 
\end{proof}
\begin{remark}[Backward stochastic differential equation]
Suppose there are no reset points, and let $V := \sum_i I^i V^i$. Integration by parts and~\eqref{ThieleEq} then yield the backward stochastic (differential) equation
\begin{align*}
V(\mathrm{d}t) = V(t-)\Phi(\mathrm{d}t) - C(\mathrm{d}t) + \sum_{i,j : i \neq j} (b^{ij}(t) + V^j(t)-V^i(t)) M^{ij}(\mathrm{d}t), \quad V(T) = 0,
\end{align*}
where $M = (M^{ij})_{i \neq j}$ is the collection of martingales given by
\begin{align*}
M^{ij}(\mathrm{d}t) := N^{ij}(\mathrm{d}t) - I^i(t-)\Lambda^{ij}(\mathrm{d}t), \quad t\geq0, 
\end{align*}
confer with Proposition~\ref{PropMartingaleProperty}. In the absolutely continuous case, this can be compared with Equations~(3.4)--(3.5) in~\cite{ChristiansenDjehiche2020}, the main difference being that our equation holds surely.
\end{remark}

\section{Comparison theorems}\label{sec:comparison}

We say that   $\Lambda$ and $\bar \Lambda$ have  identical reset points if
\begin{align*}
\{(t,\omega,i) : \Lambda^{i\largecdot}(t-)(\omega)=\infty \}= \{(t,\omega,i): \bar\Lambda^{i\largecdot}(t-)(\omega)=\infty \}.
\end{align*}
\begin{corollary}[Stochastic Cantelli theorem]
	Let two canonical insurance models $(\alpha, \Lambda, \Phi, B, b)$ and $(\bar \alpha, \bar \Lambda, \Phi, \bar B, \bar b)$  be given. Let   $\Lambda$ and $\bar \Lambda$ have  identical reset points. Suppose that the state-wise prospective reserves $(V^i)$ of the first model satisfy
	\begin{align}\label{eq:cantelli_nonMarkov}
	\begin{split}
			&\mathds{1}_{\{  \bar \Lambda^{i\largecdot}(t-)<\infty \}} I^i(t-)  \bigg(  B^i(\d t)+ \sum_{j : j\neq i}(b^{ij}(t) + V^j(t)-V^i(t)) \Lambda^{ij}(\d t)\bigg)\\
			& =\mathds{1}_{\{  \bar \Lambda^{i\largecdot}(t-)<\infty \}} I^i(t-)  \bigg(  \bar B^i(\d t)+ \sum_{j : j\neq i}(\bar b^{ij}(t) + V^j(t)-V^i(t)) \bar \Lambda^{ij}(\d t)\bigg).
			\end{split}
	\end{align}
	Then the two models have  identical state-wise prospective reserves:
	 $$(V^i)_{i \in \mathcal{Z}}=(\bar V^i)_{i \in \mathcal{Z}}.$$
\end{corollary}
\begin{remark}
In line with actuarial tradition, see also below, Corollary~8.1 is stated for non-differing interest rates. If instead $\bar \Phi \neq \Phi$, we could replace~\eqref{eq:cantelli_nonMarkov} by
\begin{align*}
			&\mathds{1}_{\{  \bar \Lambda^{i\largecdot}(t-)<\infty \}} I^i(t-)  \bigg(B^i(\d t) -V^i(t-)\Phi^i(\d t) + \sum_{j : j\neq i}(b^{ij}(t) + V^j(t)-V^i(t)) \Lambda^{ij}(\d t)\bigg)\\
			& =\mathds{1}_{\{  \bar \Lambda^{i\largecdot}(t-)<\infty \}} I^i(t-)  \bigg(\bar B^i(\d t) -V^i(t-)\bar \Phi^i(\d t) + \sum_{j : j\neq i}(\bar b^{ij}(t) + V^j(t)-V^i(t)) \bar \Lambda^{ij}(\d t)\bigg)
\end{align*}
and still obtain the same conclusion.
\end{remark}
\begin{example}[Markov model]
In continuation of Example~\ref{ex:Markov_Thiele}, suppose that $\Lambda$, $\Phi$, $B$, and $b$ as well as $\bar\Lambda$, $\bar B$, and $\bar b$ are deterministic. In that case, the state-wise prospective reserves are also deterministic, and Equation~\ref{eq:cantelli_nonMarkov} reads
\begin{align*}
			&\mathds{1}_{\{  \bar \Lambda^{i\largecdot}(t-)<\infty \}} \bigg(  B^i(\d t)+ \sum_{j : j\neq i}(b^{ij}(t) + V^j(t)-V^i(t)) \Lambda^{ij}(\d t)\bigg)\\
			& =\mathds{1}_{\{  \bar \Lambda^{i\largecdot}(t-)<\infty \}} \bigg(  \bar B^i(\d t)+ \sum_{j : j\neq i}(\bar b^{ij}(t) + V^j(t)-V^i(t)) \bar \Lambda^{ij}(\d t)\bigg).
\end{align*}
If there are no reset points, this reduces to
\begin{align*}
B^i(\d t)+ \sum_{j : j\neq i}(b^{ij}(t) + V^j(t)-V^i(t)) \Lambda^{ij}(\d t)
=
\bar B^i(\d t)+ \sum_{j : j\neq i}(\bar b^{ij}(t) + V^j(t)-V^i(t)) \bar \Lambda^{ij}(\d t),
\end{align*}
compare with Theorem~6.2 in~\cite{MilbrodtStracke1997} and Satz~10.29 in\linebreak \cite{MilbrodtHelbig2008}.
\end{example}
\begin{proof}
 The assumption that  $\Lambda$ and $\bar \Lambda$ have  identical reset points implies that $\mathcal{Y}(\Lambda)= \mathcal{Y}(\bar\Lambda)$.  The equation for  $(V^i)$ implies that $(V^i)_{i \in \mathcal{Z}}$ solves the stochastic Thiele equation of $(\bar V^i)_{i \in \mathcal{Z}}$, so that Theorem  \ref{stochThiele} yields that $V^i(t)= \bar{\E}_t^i [\bar L(t)] = \bar V^i(t)$ for all $t \geq 0$, $i \in \mathcal{Z}$. 
\end{proof}
The next result provides sufficient conditions on the transition and interest rates to ensure that the state-wise prospective reserves are on the safe side. Safe-side calculations have hitherto been explored in survival models~\cite{Lidstone1905,Norberg1985}, Markov models~\cite{Hoem1988,RamlauHansen1988,Linnemann1993}, and semi-Markov models~\cite{Niemeyer2015}.
\begin{theorem} Let  $(V^i)$ and  $(\bar V^i)$  be the state-wise prospective reserves of canonical insurance models $(\alpha, \Lambda, \Phi, B, b)$ and $(\bar \alpha, \bar \Lambda, \bar \Phi, B, b)$, where on each compact interval the differences {$\bar\Lambda - \Lambda$ shall} have finite variation. Let $ V :=\sum_i I^i V^i $, and let $R^{ij} := b^{ij} + V^j - V^i$.
\begin{itemize}
	\item[(a)]  \textbf{Pessimistic actuarial  basis:} If the differences $\bar\Lambda - \Lambda$ and $\bar\Phi - \Phi$ satisfy  
	\begin{align*}
		\mathds{1}_{\{ R^{ij}(t) \geq 0\}} \big( \bar \Lambda^{ij}-\Lambda^{ij}\big) (\d t) &\geq 	0,\\
		\mathds{1}_{\{ R^{ij}(t) \leq 0\}}  \big( \bar \Lambda^{ij}-\Lambda^{ij}\big) (\d t)&\leq 	0,\\
		\mathds{1}_{\{ V^i(t) \geq 0\}}  \big( \bar \Phi^i  -\Phi^i  \big)(\d t) &\leq 	0,\\
		\mathds{1}_{\{V^i(t) \leq 0\}}  \big( \bar \Phi^i  -\Phi^i  \big) (\d t)&\geq 	0
	\end{align*}
	 for    $t \geq 0$, $i,j \in \mathcal{Z}$, $i\neq j$, then it  holds that 
	$$ \bar V^i(t) \geq  V^i(t),  \quad t \geq 0, \, i \in \mathcal{Z}.$$ 
		
	\item[(b)] \textbf{Optimistic actuarial basis:} If the differences $\bar\Lambda - \Lambda$ and $\bar\Phi - \Phi$ satisfy 
\begin{align*}
	\mathds{1}_{\{ R^{ij}(t) \geq 0\}} \big( \bar \Lambda^{ij}-\Lambda^{ij}\big)(\d t) &\leq 	0,\\
	\mathds{1}_{\{ R^{ij}(t) \leq 0\}}  \big( \bar \Lambda^{ij}-\Lambda^{ij}\big)(\d t) &\geq 	0,\\
	\mathds{1}_{\{ V^i(t) \geq 0\}}  \big( \bar \Phi^i -\Phi^i  \big)(\d t) &\geq 	0,\\
	\mathds{1}_{\{V^i(t) \leq 0\}}  \big( \bar \Phi^i  -\Phi^i  \big)(\d t)  &\leq 	0
\end{align*}
for 	 $t \geq 0$, $i,j \in \mathcal{Z}$, $i\neq j$, then it holds that 
$$ \bar V^i(t) \leq  V^i(t),  \quad t \geq 0, \, i \in \mathcal{Z}.$$  	
\end{itemize}	
\end{theorem}
\begin{remark}
Loosely speaking, the differences $\bar\Lambda - \Lambda$ have finite variation on compact intervals if the transition rates not only have identical reset points, but if also their difference locally around each reset point is {bounded}.
\end{remark}
\begin{proof}
	Let $W^i:=  V^i - \bar V^i$, $i \in \mathcal{Z}$. 
	By taking the difference of the stochastic Thiele equations of $\bar V^i$ and $ V^i$, we obtain
	\begin{align*}
		0=\mathds{1}_{\{   \bar \Lambda^{i\largecdot}(t-)<\infty \}} I^i(t-)\Big(&W^i(\d t )  - W^i(t) \bar \Phi^i(\d t)+  \bar V^i(t-) (\bar\Phi^i -\Phi^i) (\d t) \\ &+ \sum_{j:j \neq i} ( W^j(t) - W^i(t)) \bar\Lambda^{ij}(\d t) 
		+ \sum_{j:j \neq i}  R^{ij}(t) ( \Lambda^{ij}-\bar \Lambda^{ij})(\d t) \Big),
	\end{align*}
	using the fact that $\Lambda$ and $\bar \Lambda$ have identical reset points. We can rewrite the latter equation as
	\begin{align*}
		0= \mathds{1}_{\{  \bar \Lambda^{i\largecdot}(t-)<\infty \}} I^i(t-)\Big(W^i(\d t )  + A^i(\d t)- W^i(t) \Phi^i(\d t) + \sum_{j:j \neq i} ( W^j(t) - W^i(t)) \bar\Lambda^{ij}(\d t), 
		\Big)
	\end{align*}
		for 
		$$ A^i(t):=   V^i(t-) (\bar\Phi^i - \Phi^i) (\d t) + \sum_{j:j \neq i}  R^{ij}(t) ( \Lambda^{ij}- \bar\Lambda^{ij})(\d t).$$
	which, according to Theorem~\ref{stochThiele},  has the unique solution 
	\begin{align*}
		W^i(t) =  \bar{\mathbb{E}}_t^i\bigg[ \sum_j \int_{(t,T]} \frac{\kappa (t)}{\kappa (u)} I^j(u) A^j (\d u) \bigg].
	\end{align*}
	Note  here that $C(t):= \sum_i I^i(t-)A^i(\d t)$ is an element of $\mathbb{Y}$, since $ V^i$ and $R^{ij}$ are bounded on finite intervals, see Theorem~\ref{stochThiele}(b), and since the variations of $ \bar \Phi^i - \Phi^i$ and $ \Lambda^{ij}- \bar \Lambda^{ij}$ are bounded on finite intervals by assumption. 

	Under the assumptions in part (a) of the theorem, we have 
	$$\sum_j \int_{(t,T]} \frac{\kappa (t)}{\kappa (u)} I^j(u) A^j (\d u) \geq 0,$$
	so that $	W^i(t)\geq 0$, which means that  $\bar V^i(t) \geq V^i(t)$. This argument applies for any $t \in [0,T]$ and $i \in \mathcal{Z}$. For $t >T$, both prospective reserves are simply zero because of $L(t)=0$, so that the inequality still applies.  This verifies statement (a). The proof of statement (b) is analogous.
\end{proof}
Invariances of state-wise prospective reserves w.r.t.\ changes in the insurance model are important for various reasons. This includes the development of efficient computational schemes, but perhaps plays an even more fundamental role in resolving circular model definitions arising from the specification of cash flows in terms of state-wise prospective reserves. Implicit policyholder options such as the free-policy option or the surrender (lapse) option are typical features in the design of life insurance products~\cite{Gatzert2009} that lead to such reserve dependence. 

The next result provides a collection of important invariances related to, among other things, linearly reserve-dependent and shortened cash flows. Even for Markov models, such invariances are part of the actuarial folklore and rarely given a rigorous treatment; linearly reserve-dependent and scaled cash flows constitute an exception, confer with~\cite{ChristiansenDenuitDhaene2014} and~\cite{Furrer2022}, respectively.
\begin{theorem}\label{CompTheorem1}
Let  $(V^i)_{i\in \mathcal{Z}}$ be the state-wise prospective reserves of  a canonical insurance model $(\alpha, \Lambda, \Phi, B, b)$.  Let  $\mathcal{Z}_0 \subset \mathcal{Z}$  and $\mathcal{Z}_1 = \mathcal{Z} \setminus \mathcal{Z}_0$ be any decomposition of the state space. 
	\begin{itemize}
		\item[(a)] \textbf{Irrelevant initial distribution:}   $(V^i)_{i\in \mathcal{Z}}$ is invariant with respect to  $\alpha$. 
		\item[(b)] \textbf{Irrelevant states:} Suppose that $\Lambda^{ij}=0$ for all $i \in \mathcal{Z}_0$, $j \in \mathcal{Z}_1$. Then $(V^i)_{i\in \mathcal{Z}_0}$ is invariant with respect to $b^{ij}$, $i \in \mathcal{Z}_0$, $j \in  \mathcal{Z}_1$, and $B^i$, $b^{ij}$, $\Lambda^{ij}$, $i \in \mathcal{Z}_1$, $j \in  \mathcal{Z}$, $i \neq j$.
	\item[(c)] \textbf{Irrelevant transition rate:} Suppose that  $b^{kl}+V^k-V^l=0$. Then $(V^i)_{i\in \mathcal{Z}}$  is invariant with respect to any changes of  $\Lambda^{kl}$  that preserve the reset points of $\Lambda^{k \largecdot}$. 
	\item[(d)]  \textbf{Shortened cash-flow:} Suppose that $\Lambda^{ji}=0$ for all $i \in \mathcal{Z}_0$, $j \in \mathcal{Z}_1$. 
	 Then $(V^i)_{i\in \mathcal{Z}_0}$ is invariant with respect to the modification 
	 \begin{align*}
	 	\bar b^{ij} &= b^{ij}+V^j,\quad i \in \mathcal{Z}_0,\, j \in \mathcal{Z}_1,\\
	 	\bar B^j &=0, \quad  j \in \mathcal{Z}_1, \\
	 	 \bar b^{ij} &= 0, \quad  i,j \in \mathcal{Z}_1, \, i \neq j .
	 \end{align*}
	\item[(e)]  \textbf{Cemeteries:} Suppose that $\Lambda^{ji}=0$, $b^{ij} = 0$, $B^j = 0$, and $b^{jk} = 0$ for all $i\in\mathcal{Z}_0$, $j,k \in \mathcal{Z}_1$, $k \neq j$. Further, suppose that $\Delta \Lambda^{i\ell} \leq 0$ for all $i,\ell \in \mathcal{Z}_0$, $\ell \neq i$, and $\Delta \Lambda^{ij} \geq 0$ for all $i\in\mathcal{Z}_0$, $j \in \mathcal{Z}_1$. Then $(V^i)_{i\in\mathcal{Z}_0}$ is invariant with respect to the changes {
	\begin{align*}
	\bar{\Phi}^i(\mathrm{d}t) &= \Phi^i(\mathrm{d}t) + \big(1+\Delta R^i(t)\big) \sum_{j \in \mathcal{Z}_1}\Lambda^{ij}(\mathrm{d}t), \quad  i \in \mathcal{Z}_0,\\
	\bar{\Lambda}^{ij} &= 0, \quad j\in\mathcal{Z}_1, \quad  i \in \mathcal{Z}_0,\\
	\bar{B}^i(\mathrm{d}t) &= B^i(\mathrm{d}t) - \Delta B^i(t)  \sum_{j \in \mathcal{Z}_1} \Lambda^{ij}(\mathrm{d}t), \quad  i \in \mathcal{Z}_0.
	\end{align*}}
	 		\item[(f)] \textbf{Reserve-dependent payments:}  {Suppose for $k,l \in \mathcal{Z}$, $k \neq l$, that
		 	\begin{align*} 
			b^{kl}&=a_0+ a_1 (V^k -V ^l), \\
			B^k(\d t) &= A_0(\d t ) +  V^k(t-) A_1(\d t)
			\end{align*}
			 for $\mathcal{F}^-$-adapted, jointly measurable processes $a_0, a_1, A_0,A_1$ such that
			 \begin{itemize}
			 \item $0 \leq a_1 \leq c_1 < 1$ and $0 \leq a_0 \leq c_2$ for deterministic constants $c_1$ and $c_2$,
			 \item $A_0,A_1 \in \mathbb{Y}$, $\Delta (\Phi - A_1) > -1$, and $\inf_{\omega \in \Omega} \big(R(\omega)-A_1(\omega)\big)$ is bounded on compacts.
			 \end{itemize}
			 Then $(V^i)_{i\in \mathcal{Z}}$  is invariant with respect to the changes
			 \begin{align*} 
			 \bar b^{kl}&=a_0/(1-a_1),\\
			 \bar B^k &= A_0,\\
			 \bar \Lambda^{kl}(\d t) &=(1-a_1(t)) \Lambda^{kl}(\d t),\\
			  \bar \Phi^k &=  \Phi^k - A_1. 
			\end{align*}}
	\end{itemize}
\end{theorem}
\begin{proof}\leavevmode
	\begin{itemize}
		\item[(a)] The initial distribution $\alpha$ is not needed in Theorem \ref{stochThiele}(b), so it is irrelevant. 
		\item[(b)] 	For the sub-process $(V^i)_{i\in \mathcal{Z}_0}$, the equation in Theorem  \ref{stochThiele}(b) does not depend on  $b^{ij}$, $i \in \mathcal{Z}_0$, $j \in  \mathcal{Z}_1$, and $B^i$, $b^{ij}$, $\Lambda^{ij}$, $i \in \mathcal{Z}_1$, $j \in  \mathcal{Z}$, $i \neq j$. These parameters are therefore irrelevant. 
			
		\item[(c)] If $\Lambda^{kl}$ is changed in a way that preserves the reset points of $\Lambda^{k \largecdot}$, we still have $(V^i)_{i\in \mathcal{Z}}$ as a solution of the stochastic Thiele equation.  Since solutions are unique, see Theorem  \ref{stochThiele}, we obtain the assertion. 
		\item[(d)]  By applying statement (b), we can show  that  $\bar V^i(t)=\bar{\E}_t^i [L(t)]=0$ for  $i \in \mathcal{Z}_1$, $t \geq 0$,  since only the payment processes  $\bar B^j =0$ and $\bar b^{ij} = 0$,   $i,j \in \mathcal{Z}_1$, $i \neq j$, are of relevance here. In particular,  we get
			$ \bar b^{ij} + \bar V^j -  V^i = b ^{ij}+   V^j -  V^i $ for  $i \in \mathcal{Z}_0$, $j \in \mathcal{Z}_1$, so that  $((V^i)_{i\in \mathcal{Z}_0}, (\bar V^i)_{i\in \mathcal{Z}_1})$ solves the stochastic Thiele equation of  $(\bar V^i)_{i\in \mathcal{Z}}$. Since solutions are unique, see Theorem  \ref{stochThiele}, we obtain the assertion.
		\item[(e)] The assumptions imply that $b^{ij} + V^j - V^i = {-} V^i$ for $j \in \mathcal{Z}_1$ and that
		\begin{align*}
		\mathds{1}_{\{  \Lambda^{i\largecdot}(t-)< \infty\}}I^i(t-)V^i(t) = \mathds{1}_{\{  \Lambda^{i\largecdot}(t-)< \infty\}}I^i(t-)\big({(1+\Delta R^i(t))}V^i(t-) + \Delta B^i(t)\big).
		\end{align*}
		Therefore, under the described changes, we still have $(V^i)_{i\in\mathcal{Z}_0}$ as a solution of the stochastic Thiele equation. The assertion follows by uniqueness in accordance with Theorem~\ref{stochThiele}.
		\item[(f)]  Under the described changes, we still have $(V^i)_{i\in \mathcal{Z}}$ as a solution of the stochastic Thiele equation.  Since solutions are unique, see Theorem~\ref{stochThiele}, we obtain the assertion. 		\qedhere
	\end{itemize}
\end{proof}
\begin{remark}
The invariance of Theorem~\ref{CompTheorem1}(f) has also been verified for absolutely continuous {models} in~\cite{ChristiansenDjehiche2020}, but under the requirement of domination between probability measures. While our approach here is based on canonical constructions and elementary techniques, the approach in~\cite{ChristiansenDjehiche2020} is based on backward stochastic differential equations.
\end{remark}

\begin{example}
In~\cite{Furrer2022}, the representation and computation of so-called scaled cash flows is studied. The basic setup is as follows. Consider $\mathcal{Z}_0 \subset \mathcal{Z}$, let $\mathcal{Z}_1 = \mathcal{Z} \setminus \mathcal{Z}_0$, and suppose both $\mathcal{Z}_0$ and $\mathcal{Z}_1$ are non-empty. Further, suppose that $\Lambda^{ji} = 0$ for all $i \in \mathcal{Z}_0$, $j\in\mathcal{Z}_1$, so that $\mathcal{Z}_1$ is absorbing. Let $\tilde{C} \in \mathbb{Y}$ with canonical representation $(\tilde{B},\tilde{b})$. Denote by $\tau$ the first hitting time of $\mathcal{Z}_1$ by $Z$. The scaled payments of interest are given by
\begin{align*}
{C}(\d t)
=
\rho(\tau,Z_\tau)^{\mathds{1}_{\{\tau \leq t\}}} \tilde{{C}}(\d t),
\end{align*}
where each factor $t \mapsto \rho(t,j)$, $j \in \mathcal{Z}_1$, is assumed to be $\mathcal{F}^-$-adapted, strictly positive, and below one. The main result of~\cite{Furrer2022} is an invariance result for $(V_i)_{i \in \mathcal{Z}_0}$, which retains the canonical payments $(\tilde{B},\tilde{b})$ at the cost of an adjustment to the (cumulative) transition rates and the introduction of an artificial cemetery state.

This invariance could also have been derived based on the present results, namely the uniqueness of the stochastic Thiele equation. The details of the argument uses similar techniques as the proof of Theorem~\ref{CompTheorem1} and is therefore omitted. However, at least under absolutely continuous modeling, the introduction of an artificial cemetery state in~\cite{Furrer2022} may be avoided; this is directly related to Theorem~\ref{CompTheorem1}(d). Consequently, the assumption that the scaling factors be bounded by one is obsolete.
\end{example}

\section*{Acknowledgements}

Parts of the research presented in this paper were carried out while Christian Furrer was a Junior Fellow at the Hanse-Wissenschaftskolleg Institute for Advanced Study in Delmenhorst, Germany.

\bibliographystyle{apalike}
\bibliography{references.bib}

\appendix

\end{document}